\documentclass[twoside, 11pt]{article}
\usepackage{mathrsfs}

\usepackage{mathrsfs,amsfonts,amsmath}
\RequirePackage{ifthen,calc}
\usepackage{theorem}
\usepackage[dvips]{color}

\renewcommand{\baselinestretch}{1.0}

 \catcode`@=11
 \def\@evenhead{\hbox to\textwidth{\footnotesize\rm\thepage \hfill
  {\it }}} 

 \def\@oddhead{\hbox to \textwidth{\footnotesize{\it
 } \hfill\thepage}}

 \renewcommand{\section}{\makeatletter
 \renewcommand{\@seccntformat}[1]{{\csname the##1\endcsname.}\hspace{0.20em}}
 \makeatother \@startsection
{section}
{1}
{0pt}
{0.25\baselineskip}
{0.25\baselineskip}
{\normalsize\bfseries\mathversion{bold}}}

\renewcommand{\subsection}{\makeatletter
\renewcommand{\@seccntformat}[1]{{\csname the##1\endcsname.}\hspace{0.20em}}
\makeatother \@startsection
{subsection}
{1}
{0pt}
{0.125\baselineskip}
{0.125\baselineskip}
{\normalsize\bfseries\mathversion{bold}}}

\catcode`@=12
\newcommand\ack{\section*{Acknowledgement}}

\newtheorem{thm}{\noindent Theorem}[section]
\newtheorem{lem}{\noindent Lemma}[section]
\newtheorem{cor}{\noindent Corollary}[section]
\newtheorem{prop}{\noindent Proposition}[section]

\theoremheaderfont{\normalfont\bfseries}
\theorembodyfont{\slshape} {\theorembodyfont{\rmfamily}} {\theorembodyfont{\rmfamily}
\newtheorem{defn}{\noindent Definition}[section]}
{\theorembodyfont{\rmfamily} } {\theorembodyfont{\rmfamily}
}
{\theorembodyfont{\rmfamily} } {\theorembodyfont{\rmfamily}
\newtheorem{rem}{\noindent Remark}[section]}
{\theorembodyfont{\rmfamily} } {\theorembodyfont{\rmfamily} } {\theorembodyfont{\rmfamily}
}
 \def\beqlb{\begin{eqnarray}}\def\eeqlb{\end{eqnarray}}
 \def\beqnn{\begin{eqnarray*}}\def\eeqnn{\end{eqnarray*}}
 \numberwithin{equation}{section}
 \def\d{\mathrm{d}}\def\e{\mathrm{e}}

\def
\R{{\mathbb R}}

\def\C{{\mathbb C}}
\def\qed{\hfill$\square$\smallskip}
\def\E{{\mathbb{E}}}
\def\P{{\mathbb{P}}}
\def\l{\langle}
\def\r{\rangle}

\def\proof{\noindent{\it Proof}}

\begin{document}
\title{\bf Multivariate Operator-Self-Similar Random Fields}
\author{
Yuqiang Li \thanks{Research is supported by the National Natural
Science Foundation of China (No: 10901054) }
\\ \small East China Normal
University,
\\ \\ Yimin Xiao
\thanks{Research of Y. Xiao is partially
supported by the NSF grant DMS-1006903. }
\\ \small Michigan State University }

\maketitle

\begin{abstract}
\noindent
Multivariate random fields whose distributions are invariant
under operator-scalings in both time-domain and state space are
studied. Such random fields are called operator-self-similar
random fields and their scaling operators are characterized.
Two classes of operator-self-similar stable random fields
$X=\{X(t), t \in \R^d\}$  with values in $\R^m$
are constructed by utilizing homogeneous functions and stochastic
integral representations.

\end{abstract}

\vspace{2mm}

\renewcommand{\baselinestretch}{1.1}

\section{Introduction}

A self-similar process $X=\{X(t), t\in \R\}$ is a stochastic
process whose finite-dimensional distributions are invariant under
suitable scaling of the time-variable $t$ and the corresponding
$X(t)$ in the state space. It was first studied rigorously by
Lamperti \cite{L62} under the name ``semi-stable" process. Recall
that an ${\mathbb R}^m$-valued process $X$ is called self-similar
if it is stochastically continuous (i.e. continuous in probability
at each $t \in \R$) and for every constant $r>0$, there exist a positive
number $b(r)$ and a vector $a(r)\in{\mathbb R}^m$ such that
\beqlb\label{s1-1}
\{X(rt), t \in \R\}\stackrel{d}{=}\{b(r)X(t)+a(r), t \in \R\},
\eeqlb
where $\stackrel{d}{=}$ means equality of all finite-dimensional
distributions. Lamperti \cite{L62} showed that if $X$ is \emph{proper}
(see below for the definition) then $b(r) = r^H$ for some $H
\ge 0$, which is called the self-similarity index or the Hurst
index in the literature.

Self-similar processes have been under extensive investigations
during the past four decades due to their theoretical
importance (e.g. they often arise in functional limit theorems)
and their applications as stochastic models in a wide range of
scientific areas including physics, engineering, biology,
insurance risk theory, economics, mathematical finance, just to
mention a few.

The notion of self-similarity has been extended in two ways.
The first extension is to allow scaling in the state space
${\mathbb R}^m$ by linear operators (namely, $b(r)$ in (\ref{s1-1})
is allowed to be a linear operator on ${\mathbb R}^m$) and the corresponding
processes are called operator-self-similar processes in the literature.
More specifically, Laha and Rohatgi \cite{LR82} first extended Lamperti's
notion of self-similarity by allowing $b(r)$ in (\ref{s1-1}) to be in the
set of nonsingular positive-definite self-adjoint linear operators on $\R^m$.
Hudson and Mason \cite{HM82} subsequently allowed $b(r)$ to be an arbitrary
linear operator on $\R^m$. The operator-self-similarity defined by Sato
\cite{S91} has an additional assumption that $a(r)\equiv 0$ in (\ref{s1-1}).
Thus the operator-self-similarity in the sense of Sato \cite{S91} is stronger
than that in Hudson and Mason \cite{HM82}. Various examples of operator-self-similar
Gaussian and non-Gaussian processes have been constructed and studied by
Hudson and Mason \cite{HM82}, Sato \cite{S91}, Maejima and Mason \cite{MM94},
Mason and Xiao \cite{MX00}, Didier and Pipiras \cite{DP09}. The aforementioned
extensions to operator-self-similarity is useful for establishing functional
limit theorems for multivariate time series and their statistical inference
\cite{Robinson}.

The second extension is for random fields (i.e., multi-parameter
stochastic processes) which is to allow scaling by linear operators
on the multiparameter``time''-variable $t \in {\mathbb R}^d$. This was
done by Bierm\'e, Meerschaert and Scheffler \cite{BMS07}. In their
terminology, a real-valued random field $X = \{X(t),\, t \in \R^d\}$ is
called operator-scaling if there
exist a linear operator $E$ on $\R^d$ with positive real parts of the
eigenvalues and some constant $\beta > 0$ such that for all constant $r>0$,
\begin{equation}\label{Eq:OPSS2}
\big\{ X(r^E\,t),\, t \in \R^d\big\} \stackrel{d}{=}
\big\{r^{\beta}\,X(t),\, t \in \R^d\big\}.
\end{equation}
In the above and in the sequel, $r^E$ is the linear operator on $\R^d$
defined by $r^E = \sum_{n=0}^\infty \frac{(\ln r)^n E^n} {n!}.$
A typical example of Gaussian random fields satisfying
(\ref{Eq:OPSS2}) is fractional Brownian sheets introduced by
Kamont \cite{Kamont96} and other examples have been constructed in
\cite{BMS07, Xiao09a}. We mention that (\ref{Eq:OPSS2}) leads to
\emph{anisotropy in the ``time''-variable $t$}, which is a distinct
property from those of one-parameter processes. Several authors have
proposed to apply such random fields for modeling phenomena in
spatial statistics, stochastic hydrology and imaging processing
(see \cite{BE03,BMB06,DH99}).

In this paper, we further extend the notions of operator-self-similarity
and operator-scaling to multivariate random fields by
combining the aforementioned two approaches. That is, we will
allow scaling of the random field in both ``time''-domain and
state space by linear operators. This is mainly motivated by the
increasing interest in \emph{multivariate random field models} in
spatial statistics as well as in applied areas such as
environmental, agricultural, and ecological sciences, where multivariate
measurements are performed routinely. See
Wackernagel \cite{Wack98}, Chil\'es and Delfiner \cite{CD99} and
their combined references for further information. We also believe
that the random field models constructed by Zhang \cite{Zhang07},
Gneiting, Kleiber and Schlather \cite{GKS09}, Apanasovich and Genton
\cite{AG10} are locally operator-self-similar and their tangent fields are
operator-self-similar in the sense of Definition \ref{Def:OSSF}
below. This problem will be investigated in a subsequent paper.


Throughout this paper, let $X=\{X(t), t\in{\mathbb R}^d\}$ be a random
field with values in ${\mathbb R}^m$, where $d\ge 2$ and $m\ge 2$ are
fixed integers. In the probability literature, $\R^d$ is often referred
to as the ``time''-domain (or parameter space), $\R^m$ as the state space
and $X$ a $(d, m)$-random field. We will be careful not to confuse the
terminology with the space-time random fields in geostatistics.

The following definition is a natural extension of the wide-sense
operator-self-similarity and operator-self-similarity in Sato \cite{S91}
for one-parameter processes to $(d, m)$-random fields.

\begin{defn}\label{Def:OSSF}
{\it Let $E$ be a $d\times d$ matrix whose eigenvalues have positive
real parts. A $(d, m)$-random field $X=\{X(t), t\in{\mathbb R}^d\}$
is called wide-sense operator-self-similar (w.o.s.s.) with time-variable
scaling exponent $E$, if for any constant $r>0$ there exist an $m\times m$
matrix $B(r)$ (which is called a state space scaling operator) and a
function $a_r(\cdot): {\mathbb R}^d \to
{\mathbb R}^m$ (both $B(r)$ and $a_r(\cdot)$ are non-random) such that
\beqlb\label{s1-2}
\big\{X(r^Et), \, t \in \R^d\big\}\stackrel{d}{=}
\big\{B(r)X(t)+a_r(t), t \in \R^d\big\}.
\eeqlb
If, in addition, $a_r(t)\equiv 0$, then $X$ is called operator-self-similar
(o.s.s.) with scaling exponent $E$.
}
\end{defn}

\begin{rem}
Here are some remarks about Definition \ref{Def:OSSF}.
\begin{itemize}
\item[(i)]\,
If a random field $X$ is w.o.s.s., then the consistency in \eqref{s1-2}
implies
\begin{equation}\label{Eq:Con1}
B(r_1r_2)= B(r_1) B(r_2) = B(r_2) B(r_1), \qquad \forall\, r_1, r_2 > 0
\end{equation}
and for all $r_1, r_2 > 0$ and $t \in \R^d$,
\begin{equation}\label{Eq:Con2}
\begin{split}
a_{r_1r_2}(t) = B(r_1) a_{r_2}(t) + a_{r_1}(r_2^Et)
= B(r_2) a_{r_1}(t) + a_{r_2}(r_1^E t).
\end{split}
\end{equation}

\item[(ii)]\, One can also define operator-self-similarity for random
fields by extending the analogous notion in Hudson and Mason
\cite{HM82}. Namely, we say that a $(d, m)$-random field $X=\{X(t), t\in
{\mathbb R}^d\}$ is operator-self-similar (o.s.s.) in
the sense of Hudson and Mason with time-variable
scaling exponent $E$, if for any constant $r>0$ there exist an $m\times m$
matrix $B(r)$ and a vector $a(r)\in {\mathbb R}^m$
such that
\beqlb\label{s1-HM}
\big\{X(r^Et), \, t \in \R^d\big\}\stackrel{d}{=}
\big\{B(r)X(t)+ a(r), t \in \R^d\big\}.
\eeqlb
Since the function $a(r)$ does not depend on $t \in \R^d$, (\ref{s1-HM}) is
stronger than w.o.s.s. in Definition \ref{Def:OSSF}, but is weaker than the
operator-self-similarity.
\end{itemize}
\end{rem}

Recall that a probability measure $\mu$ on $\R^m$ is \emph{full}
if its support is not contained in any proper hyperplane in $\R^m$.
We say that a $(d, m)$-random field $X=\{X(t),\, t\in\R^d\}$ is
\emph{proper} if for each $t\not=0$, the distribution of $X(t)$
is full. Then one can verify (see e.g. \cite[p.282]{HM82}) that for a
proper w.o.s.s. random field, its space-scaling operator $B(r)$ must be
nonsingular for all $r>0$.

We remark that proper w.o.s.s. random fields are special cases of
\emph{group self-similar processes} introduced by Kolody\'{n}ski
and Rosi\'{n}ski \cite{KR03} and can be studied by using their
general framework. To recall their definition, let $G$ be a group
of transformations of a set $T$ and, for each $(g, t) \in G\times T$,
let $C(g, t): \R^m\to \R^m$ be a bijection such that
\[
C(g_1g_2, t) = C(g_1, g_2(t))\circ C(g_2, t), \qquad \forall g_1, g_2 \in G
\hbox{ and } t \in T,
\]
and $C(e, t)= I$. Here $e$ is the unit element of $G$ and $I$ is the identity
operator on $\R^m$. In other words, $C$ is a cocycle for the group action
$(g, t)\mapsto g(t)$ of $G$ on $T$. According to Kolody\'{n}ski
and Rosi\'{n}ski \cite{KR03}, a stochastic process $\{X(t), t \in T\}$
taking values in $\R^m$ is called \emph{$G$-self-similar with cocycle $C$} if
\begin{equation}\label{Eq:GSS1}
\{X\big(g(t)\big),\, t\in T\}\stackrel{d}{=}\{C(g, t)X(t),\, t\in T\}.
\end{equation}
Now we take $T = \R^d$ and $G = \{r^E: r >0\}$ which is a subgroup of invertible
linear operators on $\R^d$. It is clear that if a proper $(d, m)$-random field
$X= \{X(t), t \in \R^d\}$ is w.o.s.s. in the sense of Definition
\ref{Def:OSSF}, then it is $G$-self-similar with cocycle $C$, where for each
$g= r^E \in G$ and $t \in \R^d$, $C(g,t): \R^m \to \R^m$ is defined by $C(g, t)(w)
= B(r)w + a_r(t)$. Note that $C(g,t)$ is a bijection since $X$ is proper; and it
is a cocycle because of (\ref{Eq:Con1}) and  (\ref{Eq:Con2}).

In \cite{KR03}, Kolody\'{n}ski and Rosi\'{n}ski consider a strictly stable
process $X= \{X(t), t \in T\}$ with values in $\R^m$ which is $G$-self-similar
with cocycle $C$ and characterize the minimal spectral representation of $X$
(which is a kind of stochastic integral representation and always exists for
strictly stable processes) in terms of a nonsingular action $L$  of $G$ on
a measure space $(S, {\cal B}(S), \mu)$, where $S$ is a Borel subset of a Polish
space equipped with its Borel $\sigma$-algebra ${\cal B}(S)$ and $\mu$ is a
$\sigma$-finite measure, and a cocycle $c: G\times S \to \{-1, 1\}$ relative
to $L$ (see Section 3 of \cite{KR03} for details). They also construct strictly
stable processes which are $G$-self-similar with cocycle $C$ by using nonsingular
actions $L$ of $G$ on $S$ and $\{-1, 1\}$-valued cocycle $c$ relative to $L$ (see
Section 4 of \cite{KR03}). Their general framework provides a
unified treatment for stochastic processes with various invariance properties
(such as stationarity, isotropy, and self-similarity) and is particularly powerful
when combined with methods from ergodic theory to study probabilistic and statistical
properties of $G$-self-similar strictly stable processes. See, Rosi\'{n}ski
\cite{Ros95, Ros00}, Roy and  Samorodnitsky \cite{RS08} and  Samorodnitsky
\cite{Sam04} for recent results  on stationary stable processes and random fields.
It would be very interesting to pursue further this line of research for o.s.s.
or more general $G$-self-similar stable random fields.

The main objective of the present paper is to characterize the permissible
forms for the state space scaling operator (or simply the space-scaling
operator) $B(r)$, which provides corresponding information on the cocycle
$C(g, t)$. We will also construct two types of proper o.s.s. symmetric
$\alpha$-stable $(d, m)$-random fields by using stochastic integrals of
matrix-valued deterministic functions with respect to vector-valued symmetric
$\alpha$-stable (S$\alpha$S) random measures. Our construction method is
somewhat different and less general than that of Kolody\'{n}ski and
Rosi\'{n}ski \cite{KR03} who use stochastic integrals of real-valued
deterministic functions with respect to a real-valued strictly stable
random measure and who only require their deterministic integrands to
satisfy certain recurrence equation involving a non-singular action
$L$ of $G$ on $S$ and a cocycle $c: G\times S \to \{-1, 1\}$ relative
to $L$. See Proposition 4.1 in \cite{KR03} for details. The deterministic
integrands in our
constructions are given in terms of $\Theta$-homogeneous functions (see
Definition 2.6 in \cite{BMS07} or Section 2 below). Hence the resulting
o.s.s. stable $(d, m)$-random fields in this paper are natural multivariate
extensions of the familiar linear and harmonizable fractional stable fields.
To explore the connections between these o.s.s. stable random fields and
the $G$-self-similar stable random fields in Proposition 4.1 of Kolody\'{n}ski
and Rosi\'{n}ski \cite{KR03}, we determine the non-singular action $L$ of $G =
\{r^E, r > 0\}$ on the measure space $(\R^d, {\cal B}(\R^d), \lambda_d)$
and the cocycle $c: G\times \R^d \to \{-1, 1\}$
relative to $L$ for the o.s.s. S$\alpha$S random fields constructed in
Theorems  \ref{s5-t1} and \ref{s5-t2}. These preliminary results may be
helpful for applying the
powerful tools developed in Rosi\'{n}ski \cite{Ros95, Ros00} to study
operator-self-similar S$\alpha$S random fields.

The rest of this paper is divided into three sections. In Section 2, we
provide some preliminaries and state the main results of this paper. Theorem
\ref{s2-t1} proves that, under some standard conditions, the space-scaling
operator $B(r)$ in (\ref{s1-2}) must be of the form $B(r)= r^D$ for some $D
\in M(\R^m)$, which will be called the state space scaling exponent (or the
space-scaling exponent) of $X$. Theorem \ref{s2-t2} is an analogous result for
o.s.s. random fields. Theorems \ref{s5-t1} and \ref{s5-t2} provide
general ways for constructing proper moving-average-type and harmonizable-type
o.s.s. stable $(d, m)$-random fields with prescribed operator-self-similarity
exponents. We also describe the connection between these random fields and the
$G$-self-similar stable random fields in \cite{KR03}. In Section 3 we characterize
the forms of the space-scaling operators and prove Theorems \ref{s2-t1} and
\ref{s2-t2}. The proofs of Theorem \ref{s5-t1} and \ref{s5-t2} are given in
Section 4. It will be clear that the arguments in Hudson and Mason \cite{HM82},
Maejima and Mason \cite{MM94} and Bierm\'e, Meerschaert and Scheffler
\cite{BMS07} play important roles throughout this paper.

We end this section with some notation. For any integer $n\geq 1$, we use
$\lambda_n$ to denote the Lebesgue measure on $\R^n$ and $\mathcal{B}(\R^n)$
the Borel algebra. The Euclidean norm and inner product in  $\R^n$ are denoted
by $|\cdot\,|$ and $\langle \cdot, \cdot\rangle$, respectively. Let End$(\R^n)$
be the set of all linear operators on $\R^n$ or, equivalently, $n\times n$
matrices. The set of invertible linear operators in End($\R^n$) is denoted by
Aut$(\R^n)$. Let $Q(\R^n)$ be the set of $A\in \text{Aut}(\R^n)$ such that all
eigenvalues of $A$ have positive real parts. Let $M(\R^n)$ be the set of $A\in
\text{End}(\R^n)$ such that all eigenvalues of $A$ have nonnegative real parts and
every eigenvalue of $A$ with real part equal to zero (if it exists) is a simple
root of the minimal polynomial of $A$.

We will use  $C_0, C_1, C_2, \cdots$ to denote unspecified positive finite constants
which may not necessarily be the same in each occurrence.
\bigskip

\section{Main results}

Throughout this paper, $E\in Q(\R^d)$ is a fixed $d\times d$ matrix.
$E^*$ is the adjoint of $E$; and $\alpha\in(0, 2]$ is a constant.

Our first result characterizes the form of the space-scaling
operator $B(r)$ for a w.o.s.s. random field.

\begin{thm}\label{s2-t1}
Let $X=\{X(t), t\in\R^d\}$ be a stochastically continuous and proper w.o.s.s. random
field with values in $\R^m$ and time-variable scaling exponent $E\in Q(\R^d)$.
There exist a matrix $D\in M(\R^m)$ and a function
$b_r(t): (0,\,\infty)\times\R^d \to \R^m$ which is continuous at every $(r, t)\in
(0,\,\infty)\times\R^d$ such that for all constants $r>0$
\beqlb\label{s2-t1-1}
\big\{X(r^E t),\, t \in \R^d\big\}\stackrel{d}{=}
\big\{r^D X(t)+ b_r(t), \, t \in \R^d\big\}.
\eeqlb
Furthermore, $X(0)= a$  a.s. for some constant
vector  $a \in \R^m$ if and only if $D\in Q(\R^m)$. In this latter case, we define
$b_0(t) \equiv a$ for all $t \in \R^d$, then the function $(r, t)\mapsto b_r(t)$
is continuous on $[0,\,\infty)\times\R^d$.
\end{thm}

The operator $D$ will be called the state space scaling exponent (or space-scaling
exponent). For a given time-variable scaling exponent $E\in Q(\R^d)$, the corresponding
exponent $D$ may not be unique. In order to emphasize the roles of the linear operators
$E$ and $D$, we call $X$ w.o.s.s. with exponents ($E, D$), or simply ($E, D$)-w.o.s.s.
By combining Theorem \ref{s2-t1} with Lemma 2.4 in \cite{S91}, we derive readily the
following corollary. Of course, (\ref{Eq:Con3}) also follows from (\ref{Eq:Con2}).

\begin{cor}\label{s2-c1}
Under the conditions of Theorem \ref{s2-t1}, the function ${b}_r(t)$ is uniquely
determined by $E$ and $D$. Furthermore,
\begin{equation}\label{Eq:Con3}
b_{r_1r_2}(t)= {b}_{r_1}(r_2^Et)+r_1^D {b}_{r_2}(t)={b}_{r_2}
(r_1^Et)+r_2^D {b}_{r_1}(t)
\end{equation}
for all $r_1, r_2 > 0$ and $t \in \R^d$.
\end{cor}

The following corollary expresses the function $b_r(t)$ in terms of a
function of $t$ and the scaling exponents $E$ and $D$.
\begin{cor}\label{s2-c2}
Under the conditions of Theorem \ref{s2-t1}, there exists a continuous function
$b(\cdot): \R^d \backslash \{0\} \to \R^m$ such that the function $b_r(t)$
satisfies
\begin{equation}\label{Eq:br}
b_r(t) = b(r^Et) - r^D b(t), \qquad \forall r > 0 \ \hbox{ and } t \in \R^d
\backslash \{0\}.
\end{equation}
If, in addition, $D \in Q(\R^m)$ and $X(0) = a$ a.s., where $a \in \R^m$ is a
constant vector, then we can extend the definition of $b(\cdot)$ to $\R^d$ by
defining $b(0) = a$ such that (\ref{Eq:br}) holds for all $r > 0$ and $t \in \R^d$.
\end{cor}

The proof of this corollary is based on the polar coordinate representation
of $t \in \R^d\backslash \{0\}$ under operator $E$ given in \cite[p.317]{BMS07}
(the definition is recalled below) and will be given in Section 3.

The next theorem is an analogue of Theorem \ref{s2-t1} for o.s.s. random fields.

\begin{thm}\label{s2-t2}
Let $X=\{X(t), t\in\R^d\}$ be a stochastically continuous and proper  random
field with values in $\R^m$.
\begin{itemize}
\item[(i)]\, If $X$ is o.s.s. with time-variable scaling exponent $E\in Q(\R^d)$,
then there exists a matrix $D \in M(\R^m)$ such that for all $r>0$
\beqlb\label{s2-t2-2}
\big\{X(r^E t),\, t \in \R^d\big\}\stackrel{d}{=}
\big\{r^D X(t), \, t \in \R^d\big\}.
\eeqlb
Moreover, $D \in Q(\R^m)$ if and only if $X(0) = 0$ a.s.

\item[(ii)] If $X$ is o.s.s. with time-variable scaling exponent $E\in Q(\R^d)$
in the sense of Hudson and Mason, then there exist a matrix $D \in M(\R^m)$ and
a continuous function $b(r): (0,\,\infty) \to \R^m$  such that for all constants $r>0$
\beqlb\label{s2-t2-3}
\big\{X(r^E t),\, t \in \R^d\big\}\stackrel{d}{=}
\big\{r^D X(t)+ b(r), \, t \in \R^d\big\}.
\eeqlb
\end{itemize}
\end{thm}

A $(d, m)$-random field $X=\{X(t), t \in \R^d\}$ is called
operator-self-similar with exponents ($E, D$) (or $(E, D)$-o.s.s.) if
(\ref{s2-t2-2}) holds. By Corollary \ref{s2-c2}, we see that if $X$ is
a w.o.s.s. $(d, m)$-random field as in Theorem \ref{s2-t1}, then the
$(d, m)$-random field $Y= \{X(t) - b(t), t \in \R^d\backslash \{0\}$\} is
operator-self-similar with exponents $(E, D)$. Using the
terminology of Sato \cite{S91}, we also call the function $b(t)$ in Corollary
\ref{s2-c2} the drift function of $(d, m)$-random field $X$.

Recall that a $(d, m)$-random field $X$ is said to have stationary increments
if for all $h\in\R^d$,
\beqlb\label{s2-t1-3}
\big\{X(t+h)-X(h), \, t\in \R^d\big\}\stackrel{d}{=}
\big\{X(t), \, t\in \R^d\big\}.
\eeqlb
Now we turn to construction of interesting examples of stable o.s.s. $(d, m)$-random
fields with stationary increments, by using stochastic integrals with respect to
a stable random measure. We refer to Samorodnitsky and Taqqu \cite{ST94} for a
systematic account on the latter. For simplicity we will only consider symmetric
$\alpha$-stable (S$\alpha$S) random fields and the main idea comes from \cite{BMS07},
\cite{MX00} and \cite{MM94}. By using stochastic integral with respect to a strictly
stable random measure one can extend the construction to obtain strictly stable o.s.s.
$(d, m)$-random fields. Kolody\'{n}ski and Rosi\'{n}ski \cite{KR03} use this more
general approach.

For any given operators $E \in Q(\R^d)$ and $D \in Q(\R^m)$ we construct
$(E, D)$-o.s.s. $\alpha$-stable random fields by using stochastic integrals
with respect to a symmetric $\alpha$-stable random vector measure (when
$\alpha=2$ the resulting o.s.s. random fields are Gaussian). For this purpose,
we recall briefly the definitions of stochastic integrals with respect to
vector-valued $\alpha$-stable random measures.

Let $(\Omega, \mathcal{F}, \P)$ be the underlying probability space and
let $L^0(\Omega)$ be the set of all $\R^m$-valued random vectors defined
on $(\Omega, \mathcal{F}, \P)$. Let $S_{m-1}$ be the unit sphere in $\R^m$
with the Borel algebra $\mathcal{B}(S_{m-1})$.

Let $K$ be a $\sigma$-finite measure on $\R^d\times S_{m-1}$ such
that for any $A\in \mathcal{B}(\R^d)$, $K(A\times\cdot)$ is a
symmetric finite measure on $(S_{m-1}, \mathcal{B}(S_{m-1}))$.
Denote
\[
\mathcal{M}:=\{A\in\mathcal{B}(\R^d):\, K(A, S_{m-1})<\infty\}.
\]

We first give the definition of a vector-valued symmetric
$\alpha$-stable (S$\alpha$S) random measure.

\begin{defn}\label{s4-d1}
{\it An $\R^m$-valued S$\alpha$S random measure on $(\R^d,
\mathcal{B}(\R^d))$ with control measure $K$ is an
independently scattered $\sigma$-additive $\R^m$-valued
set function $M: \mathcal{M}\to L^0(\Omega)$
such that, for every $A\in \mathcal{M}$, the random vector
$(M^1(A),\cdots, M^m(A))$ is jointly S$\alpha$S with
spectral measure $K(A, \cdot)$. Here, the meaning of
``independently scattered" and ``$\sigma$-additive" is
the same as in  Section 3.3 of \cite{ST94}.}
\end{defn}

One can apply Kolmogorov's extension theorem to show that
$\R^m$-valued S$\alpha$S random measure $M$ in
Definition \ref{s4-d1} exists, with finite-dimensional
distributions characterized by
\beqlb\label{s4-d1-1}
&&\E\exp\Bigg\{i\sum_{j=1}^k \l \theta_j, M(A_j)
\r\Bigg\}\nonumber \\
&&\qquad\quad=\exp\Bigg\{-\int_{\R^d}\int_{S_{m-1}}
\bigg|\sum_{j=1}^k\sum_{l=1}^m s_l\theta_{j,l}
{\bf 1}_{A_j}(x)\bigg|^\alpha \, K(\d x, \d s)\Bigg\},
\eeqlb
where $A_j \in\mathcal{M}$ and $\theta_j=(\theta_{j,1},\cdots,
\theta_{j,m})\in \R^m$ for all $k\geq 1$ and
$j=1,\cdots,k$.

In this paper, unless stated otherwise, the control measure $K$
will always be assumed to have the form $K(A, B)=\lambda_d(A)
\Gamma(B)$ for all $A\in \mathcal{B}(\R^d)$ and $B\in\mathcal{B}
(S_{m-1})$, where $\lambda_d$ is the Lebesgue measure on
$\R^d$ and $\Gamma(\cdot)$ is the normalized uniform measure on
$S_{m-1}$ such that for all $\theta=(\theta_1, \cdots,\theta_m)
\in\R^m$,
$$
\int_{S_{m-1}}  \bigg|\sum_{l=1}^m s_l\theta_l\bigg|^\alpha
\Gamma(\d s)= |\theta|^\alpha.
$$
Therefore, for disjoint sets $A_j
\in \mathcal{M}$, $j=1,2,\cdots, k$, Eq. (\ref{s4-d1-1}) can
be written as
\beqlb\label{s4-d1-2}
\E\exp\bigg\{i\sum_{j=1}^k \l \theta_j,\, M(A_j)
\r\bigg\}=\exp\bigg\{-\sum_{j=1}^k\lambda_d(A_j)
|\theta_j|^\alpha\bigg\}.
\eeqlb

For any real $m\times m$ matrix $Q$, let $\|Q\|:=
\max_{|x|=1}|Qx|$ be the operator norm of $Q$. It is easy
to see that for $Q_1,\ Q_2 \in$ End$(\R^m)$,
$\|Q_1Q_2\| \le \|Q_1\|\cdot \|Q_2\|$.
The following theorem is an extension of Theorem 4.1 in
\cite{MM94} and defines stochastic integrals of matrix-valued
functions with respect to a vector-valued S$\alpha$S random
measure.

\begin{thm}\label{s4-t1}
Let $\{Q(u), u\in\R^d\}$ be a family of real $m\times m$-matrices.
If $Q(u)$ is $\mathcal{B}(\R^d)$-measurable and $\int_{\R^d}
\|Q(u)\|^\alpha\d u<\infty$, then the stochastic integral
$$
I(Q):=\int_{\R^d} Q(u) \, M (\d u)
$$
is well defined and it is a symmetric $\alpha$-stable
vector in $\R^m$ with characteristic function
\beqlb\label{s4-t1-1}
\E\Big[\e^{i \l\theta, I(Q)\r}\Big]
=\exp\bigg\{-\int_{\R^d}\big|Q(u)^*\theta\big|^\alpha
\d u\bigg\}, \;\qquad \forall \, \theta\in\R^m.
\eeqlb
\end{thm}

It follows from (\ref{s4-t1-1}) and Lemma \ref{s2-l2} below that if the
matrix $Q(u)$ is invertible for $u$ in a set of positive $\lambda_d$-measure,
then the distribution of $I(Q)$ is full. This fact is useful for constructing
proper S$\alpha$S random fields.

One can also define stochastic integrals of complex matrix-valued
functions with respect to a complex vector-valued S$\alpha$S
random measure $\widetilde{M}$ defined as follows. Let $\overline{M}$
be an $\R^{2m}$-valued S$\alpha$S-random measure on $(\R^d, \mathcal{B}
(\R^d))$ with control measure $K=\lambda_d \times \Gamma$, where
$\Gamma$ is the normalized uniform measure on $S_{2m-1}$. Define the
complex-valued S$\alpha$S-random measures
$\widetilde{M}_k=\overline{M}_k+i\overline{M}_{m+k}$ for
all $k=1,\cdots,m$. Then $\widetilde{M}=(\widetilde{M}_1, \cdots,
\widetilde{M}_m)$ is a $\C^m$-valued S$\alpha$S-random measure
with control measure $K$. Its real and imaginary parts are
$\widetilde{M}_{R}=(\overline{M}_1,\cdots, \overline{M}_m)$
and $\widetilde{M}_{I}=(\overline{M}_{m+1},\cdots,\overline{M}_{2m})$,
respectively. The following theorem defines stochastic integrals of complex
matrix-valued functions with respect to $\widetilde{M}$.

\begin{thm}\label{s4-t2}
Let $\{\widetilde{Q}_1(u),\, u\in\R^d\}$ and $\{\widetilde{Q}_2(u),
\, u\in\R^d\}$ be two families of real $m\times m$-matrices. Let
$\widetilde{Q}(u)=\widetilde{Q}_1(u)+i\widetilde{Q}_2(u)$
for all $u\in\R^d$. If $\widetilde{Q}_1(u)$ and $\widetilde{Q}_2(u)$
are $\mathcal{B}(\R^d)$-measurable and $\int_{\R^d}
(\|\widetilde{Q}_1(u)\|^\alpha+\|\widetilde{Q}_2(u)\|^\alpha)
\, \d u<\infty$, then
$$
\widetilde I(\widetilde{Q}):={\rm Re}\int_{\R^d} \widetilde{Q}(u)\,
\widetilde{M}(\d u)
$$
is well defined and it is a symmetric $\alpha$-stable vector in
$\R^m$ with its characteristic function given by
\beqlb\label{s4-t2-1}
\begin{split}
&\E\Big[\e ^{i\l \theta, \widetilde I(\widetilde{Q})\r}\Big]\\
&=\exp\Bigg\{-\int_{\R^d}
\bigg(\sqrt{\big|\widetilde{Q}_1(u)^*\theta\big|^2+
 \big|\widetilde{Q}_2(u)^*\theta\big|^2}\, \bigg)^\alpha\,
\d u\Bigg\},\quad \; \forall \, \theta\in\R^m.
 \end{split}
 \eeqlb
\end{thm}

It follows from (\ref{s4-t2-1}) and Lemma \ref{s2-l2} below that if the
matrix $\widetilde Q_1(u)$ or $\widetilde Q_2(u)$ is invertible for
$u$ in a set of positive $\lambda_d$-measure,
then the distribution of $\widetilde I(\widetilde Q)$ is full.

Based on the above stochastic integrals, we can construct
moving-average type or harmonizable-type $\alpha$-stable random
fields by choosing suitable functions $Q$ and $\widetilde{Q}$. In
order to obtain o.s.s. random fields, we will make use of the
$\Theta$-homogeneous functions and the $(\beta,
\Theta)$-admissible functions as in \cite{BMS07}.

Suppose $\Theta\in Q(\R^d)$ with real parts of the eigenvalues
$0<a_1<a_2<\cdots<a_p$ for $p\leq d$. Let $q$ denote the trace of
$\Theta$. It follows from \cite[p.314]{BMS07} that every
$x\in\R^d\setminus\{0\}$ can be written uniquely as
$x=\tau(x)^\Theta l(x)$ for some radial part $\tau(x)>0$ and some
direction $l(x)\in \Sigma_0$ such that the functions $x\mapsto
\tau(x)$ and $x\mapsto l(x)$ are continuous, where
$\Sigma_0=\{x\in\R^d, \, \tau(x)=1\}$. It is well-known that
$\tau(x)=\tau(-x)$ and $\tau(r^\Theta x)=r\tau(x)$ for all $r>0$
and $x\in\R^d$. Moreover, $\Sigma_0$ is compact;
$\tau(x)\to\infty$ as $|x|\to\infty$ and $\tau(x)\to 0$ as $|x|\to0$.
In addition, Lemma 2.2 in \cite{BMS07} shows that there exists a
constant $C_0\geq 1$ such that for all $x,\, y\in\R^d$
\beqlb\label{s5-1}
\tau(x+y)\leq C_0(\tau(x)+\tau(y)).
\eeqlb
For convenience, we call $(\tau(x), l(x))$ the polar coordinates
of $x$ under operator $\Theta$. According to Definition 2.6 in
\cite{BMS07}, a function $\phi: \R^d\to\C$ is said to be
$\Theta$-homogeneous if $\phi(r^\Theta x)=r\phi(x)$ for all $r>0$
and $x\in\R^d\setminus\{0\}$. Obviously, if $\phi$ is
$\Theta$-homogeneous, continuous on $\R^d$ and takes positive
values on $\R^d\setminus\{0\}$, then $\phi(0)=0$,
\beqlb\label{s5-2}
M_{\phi}=\max_{\theta\in\Sigma_0}\phi(\theta)>0\;\ \ \text{and}\ \
\ m_{\phi}=\min_{\theta\in\Sigma_0}\phi(\theta)>0.
\eeqlb

Let $\beta>0$. Recall from Definition 2.7 in \cite{BMS07} that
a function $\psi: \R^d\to [0,\infty)$ is called $(\beta,
\Theta)$-admissible, if $\psi(x)>0$ for all $x\not=0$ and
for any $0<A<B$ there exists a positive constant $C_1>0$ such that,
for $A\leq |y|\leq B$,
$$
\tau(x)\leq 1 \ \Rightarrow \ |\psi(x+y)-\psi(y)|\leq C_1\tau(x)^\beta.
$$

For any given matrices $E \in Q(\R^d)$ and $D \in Q(\R^m)$, Theorem
\ref{s5-t1} provides a class of moving-average-type o.s.s. $\alpha$-stable
random fields with prescribed self-similarity exponents $(E, D)$.

\begin{thm}\label{s5-t1}
Suppose $\phi: \R^d\mapsto[0, \infty)$ is an $E$-homogeneous,
$(\beta, E)$-admissible function for some constant $\beta>0$. Let $q$
be the trace of $E$, $H$ be the maximum of the real parts of the eigenvalues
of $D \in Q(\R^m)$ and let $I$ be the identity operator in $\R^m$.
If $H<\beta$, then the random field
\begin{equation}\label{Eq:Xphi}
X_{\phi}(x)= \int_{\R^d}\Big[\phi(x-y)^{D-qI/\alpha}
-\phi(-y)^{D-qI/\alpha}\Big]\, M(\d y),\quad x\in\R^d
\end{equation}
is well defined, where the stochastic integral in (\ref{Eq:Xphi})
is defined as in Theorem \ref{s4-t1}. Furthermore, $X_\phi =
\{X_{\phi}(x), x\in\R^d\}$ is a stochastically
continuous $(E, D)$-o.s.s. $S \alpha S$-random field with stationary
increments.
\end{thm}

\begin{rem}
We can choose $E$ and $D$ to ensure that the S$\alpha$S-random field
$X$ is proper. A sufficient condition
is that $q/\alpha$ is not an eigenvalue of $D$. This implies that, for
every $x \in \R^d$, the operator $\phi(x-y)^{D-qI/\alpha}
-\phi(-y)^{D-qI/\alpha}$ is invertible for $y$ in a subset of $\R^d$ with
positive Lebesgue measure, which ensures that the distribution of
$X_{\phi}(x)$ is full.
\end{rem}

When $m=1$ and $D= HI$, Theorem \ref{s5-t1} reduces to Theorem 3.1 in
Bierm\'{e}, Meerschaert and Scheffler \cite{BMS07}. For a general
$D \in Q(\R^m)$, the following example of $X_{\phi}$ is instructive.
Let $E = (e_{ij})$ be the diagonal matrix in $Q(\R^d)$ with
$e_{jj}=\gamma_j^{-1}$, where $\gamma_j \in (0, 1)$ ($1\le j \le d$)
are constants. It can be verified that there exists a constant
$C_2 \ge 1$ such  that the corresponding radial part $\tau(x)$
satisfies
\begin{equation}\label{Eq:tau}
C_2^{-1} \sum_{j=1}^d|x_j|^{\gamma_j} \le \tau(x) \le C_2\,
\sum_{j=1}^d|x_j|^{\gamma_j}
\end{equation}
for all $x \in \R^d$. Note that the function $\phi(x) =
\sum_{j=1}^d|x_j|^{\gamma_j}$ is $E$-homogeneous and $(\beta, E)$-admissible
with $\beta = 1$. This latter assertion follows from (\ref{Eq:tau}) and
the elementary inequality $|x+y|^\gamma \le |x|^\gamma+|y|^\gamma$ if
$\gamma \in (0, 1)$. Let $D\in Q(\R^m)$ be as in Theorem \ref{s5-t1},
then $X_\phi = \{X_{\phi}(x), \, x\in\R^d\}$ defined by
$$
X_{\phi}(x)= \int_{\R^d}\Bigg[\bigg(\sum_{j=1}^d |x_j-y_j|^{\gamma_j}
\bigg)^{D-qI/\alpha} -\bigg(\sum_{j=1}^d |y_j|^{\gamma_j}
\bigg)^{D-qI/\alpha} \Bigg]\, M(\d y)
$$
is an $(E, D)$-o.s.s. $S \alpha S$ random field with stationary increments.
Moreover, since $H < 1$ and $q/\alpha > 1$ (we have assumed $d \ge 2$ in
this paper), we see that $X_{\phi}$ is proper.

Similarly to Theorem \ref{s5-t1}, we can construct harmonizable-type
o.s.s. $S \alpha S$ stable random fields as follows.
\begin{thm}\label{s5-t2}
Suppose $\psi: \R^d\mapsto[0, \infty)$ is a continuous,
$E^*$-homogeneous function such that $\psi(x)\not=0$ for $x\not=0$.
Let $q$ be the trace of $E$ and let $I$ be the identity operator in $\R^m$.
If $D\in Q(\R^m)$ and its maximal real part of the eigenvalues
$H< a_1$, where $a_1$ is the minimal real part of the eigenvalues of
$E$, then the random field
\begin{equation}\label{Eq:Xpsi}
\widetilde{X}_{\psi}(x)= {\rm Re}\int_{\R^d}\big(\e^{i \l x,y \r}-1\big)
\, \psi(y)^{-D-qI/\alpha}\, \widetilde{M}(\d y), \quad \forall
x\in\R^d,
\end{equation}
is well defined, where the stochastic integral in (\ref{Eq:Xpsi})
is defined as in Theorem \ref{s4-t2}. Furthermore,
$\widetilde{X}_{\psi}= \{\widetilde{X}_{\psi}(x), x\in\R^d\}$ is a
stochastically continuous, proper $(E, D)$-o.s.s. $S \alpha S$-random
field with stationary increments.
\end{thm}

\begin{rem}
Unlike in Theorem \ref{s5-t1}, $\widetilde{X}_{\psi}$ in Theorem \ref{s5-t2}
is always proper.
\end{rem}

Theorem \ref{s5-t2} is a multivariate extension of Theorem 4.1 and
Corollary 4.2 of Bierm\'{e}, Meerschaert and Scheffler \cite{BMS07}. To
give a representative of the harmonizable-type o.s.s. in Theorem \ref{s5-t2},
again we take $E= (e_{ij})
\in Q(\R^d)$ to be the diagonal matrix as above. Let $\psi(x) =
\sum_{j=1}^d|x_j|^{\gamma_j}$, which is $E^*$-homogeneous.
Then, for any $D\in Q(\R^m)$ with its maximal real parts of
the eigenvalues $H< \min\{\gamma_j^{-1}\}$, the S$\alpha$S-random
field $\widetilde{X}_{\psi} = \{\widetilde{X}_{\psi}(x), \,
x\in\R^d\}$ defined by
\begin{equation}\label{Eq:Xpsi2}
\widetilde{X}_{\psi}(x)= {\rm Re}\int_{\R^d}\frac{\e^{i \l x,y \r}-1}
{\big(\sum_{j=1}^d |y_j|^{\gamma_j}\big)^{D +qI/\alpha}}\,
\widetilde{M}(\d y)
\end{equation}
is proper and $(E, D)$-o.s.s. with stationary increments. In the special
case of $D = I$, the stable random field $\widetilde{X}_{\psi}$
has been studied in Xiao \cite{Xiao09b}. We believe that the argument
in proving Theorem 3.4 in \cite{Xiao09b} can be applied to show
that $\widetilde{X}_{\psi}$ has the property of strong local
nondeterminism, which is useful for establishing the joint continuity
of the local times of $\widetilde{X}_{\psi}$.

The o.s.s. S$\alpha$S $(d, m)$-random fields in Theorems \ref{s5-t1}
and \ref{s5-t2} provide concrete examples for the
$G$-self-similar stable random fields in Proposition 4.1 of Kolody\'{n}ski
and Rosi\'{n}ski \cite{KR03}. Recall that the o.s.s. S$\alpha$S random fields
in Theorems \ref{s5-t1} and \ref{s5-t2} are $G$-self-similar with cocycle
$C$, where $G= \{r^E, r>0\}$ and $C(r, t) = r^D$ for every $r > 0$
and $t \in \R^d$. In the following we provide non-singular actions of
$G = \{r^E, r > 0\}$ on $(\R^d, {\cal B}(\R^d), \lambda_d)$ and cocycles
$c: G\times \R^d \to \{-1, 1\}$ (or $\{z \in \C: |z| =1\}$ in the complex
case) such that the integrands in (\ref{Eq:Xphi}) and (\ref{Eq:Xpsi})
satisfy the recurrence equation (4.1)
in Kolody\'{n}ski and Rosi\'{n}ski \cite{KR03}.

For the o.s.s. S$\alpha$S random field ${X}_{\phi}$ in Theorem \ref{s5-t1},
the non-singular action of $G$ on $\R^d$ is $L_r(s) = r^{E}s$, and the
cocycle $c(r, x) \equiv 1$. A change of variable shows that
\begin{equation}\label{Eq:RN}
\frac{{\rm d} (\lambda_d\circ L_{r^{-1}})} {{\rm d}\lambda_d} = r^{-q},
\end{equation}
where $q$ is the trace of $E$. By using (\ref{Eq:RN}) and the $E$-homogeneity
of $\phi$ one can verify that the family of
integrands $\{{f}_x, x \in \R^d\}$ in Theorem \ref{s5-t1}, where
\[
f_x(y) =  \phi(x-y)^{D-qI/\alpha}
-\phi(-y)^{D-qI/\alpha}
\]
is a matrix-valued function, satisfies
\begin{equation}\label{Eq:Recurrence}
f_{r^Ex} (y) = c(r, L_{r^{-1}}y) \bigg\{\frac{{\rm d} (\lambda_d\circ L_{r^{-1}})}
 {{\rm d}\lambda_d}\bigg\}^{1/\alpha} C(r, x)f_x\circ L_{r^{-1}}(y), \qquad \forall
 y \in \R^d,
\end{equation}
which is an analogue of the recurrence equation (4.1) in Kolody\'{n}ski and
Rosi\'{n}ski \cite{KR03}.

For the o.s.s. S$\alpha$S random field $\widetilde{X}_{\psi}$ in Theorem \ref{s5-t2},
the non-singular action of $G$ on $\R^d$ is $\widetilde{L}_r(s) = r^{E^*}s$
and the cocycle $c(r, x) \equiv 1$. Then, by using (\ref{Eq:RN}) and the
$E^*$-homogeneity of $\psi$ one can verify
that the family of integrands $\{\widetilde{f}_x, x \in \R^d\}$,
where
\[
\widetilde{f}_x(y) = \big(\e^{i \l x,y \r}-1\big)\, \psi(y)^{-D-qI/\alpha},
\]
satisfies the recurrence equation  (\ref{Eq:Recurrence}) with $L$ being replaced
by $\widetilde{L}$.
\bigskip

\section{Characterization of space-scaling exponents:
Proofs of Theorems \ref{s2-t1} and \ref{s2-t2}}

In this section, we prove Theorem \ref{s2-t1}. The main idea
of our proof is originated from \cite{HM82} and \cite{S91}.
We will make use of the following lemmas which are taken
from \cite{S91} and \cite{S69}, respectively.

\begin{lem}\label{s2-l1}
\upshape{(\cite[Lemma 2.6]{S91})}
\newline\noindent For any integer $n \ge 1$, $H\in Q(\R^n)$
if and only if $\lim_{r\downarrow0} r^H x=0$ for every $x\in\R^n$.
$H\in M(\R^n)$ if and only if $\limsup_{r\downarrow
0}|r^Hx|<\infty$ for every $x\in\R^n$.
\end{lem}

\begin{lem}\label{s2-l2}
\upshape{(\cite [Proposition 1]{S69})}
\newline\noindent
A probability measure $\mu$ on $\R^n$ is not full if and only
if there exists a vector $y\in \R^n \backslash\{0\}$ such
that $|\widehat{\mu}(cy)|=1$ for all $c\in \R$, where
$\widehat{\mu}$ is the characteristic function of $\mu$.
\end{lem}

For $r>0$ and $E\in Q(\R^d)$ fixed, define $G_r$ to be the set
of $A\in \text{Aut}(\R^m)$ such that
$\{X(r^Et),\, t \in \R^d\}\stackrel{d}{=}\{AX(t)+ b(t),
\, t \in \R^d\},$ for some function $b:
\R^d \to \R^m$. Let $G=\bigcup_{r>0}G_r.$

\begin{lem}\label{s3-l1}
The set $G$ is a subgroup of $\text{Aut}(\R^m)$. In particular,
the identity matrix $I\in G_1$; $A\in G_r$ implies $A^{-1}\in
G_{1/r}$; $A\in G_r$ and $B\in G_s$ imply $AB\in G_{sr}$.
\end{lem}

\proof. This can be verified by using the above definition
and the proof is elementary. We omit the details here. \qed

\begin{lem}\label{s3-l2}
The following statements are equivalent:
\newline{\rm (1)} There exist a sequence $\{r_n, n \ge 1\}$ with
$r_n\downarrow 0$ and $A_n\in G_{r_n}$
such that $A_n$ tends to $A\in\text{Aut}(\R^m)$.
\newline{\rm (2)} $\{X(t), \,  t \in \R^d\}\stackrel{d}{=}
\{X(0)+\phi(t), \,  t \in \R^d\}$,
where $\phi$ is unique and continuous on $\R^d$.
\newline{\rm (3)} $G=G_s$ for all $s>0$.
\newline{\rm (4)} $G_s\cap G_r\not=\emptyset$ for some distinct
$s,\, r > 0$.
\end{lem}
\proof. (1)$\Rightarrow$(2). Assume (1) holds then we have
that
$\{X(r_n^Et), \,  t \in \R^d\} \stackrel{d}{=}\{A_nX(t)
+b_{r_n}(t), \, \,  t \in \R^d\}$. By Lemma
\ref{s2-l1} and the stochastic continuity of $X$, we
derive that there is a function $b(t)$ such that
$\{X(t), \,  t \in \R^d\}\stackrel{d}{=}
\{A^{-1}X(0)-A^{-1}b(t),  \, t \in \R^d\}$ and,
in  particular,
$X(0)\stackrel{d}{=}A^{-1}X(0)-A^{-1}b(0)$.
This yields (2) with $\phi(t)=A^{-1}b(0)-A^{-1}b(t)$. The
continuity of $\phi$ follows from the stochastic
continuity of $X$ and the uniqueness of $\phi$
follows from Lemma 2.4 in \cite{S91}.

(2)$\Rightarrow$(3) Suppose (2) holds and $A\in G_r$.
Then
$$
\{X(0)+\phi(r^Et), \,  t \in \R^d\}\stackrel{d}
{=}\{X(r^Et), \,  t \in \R^d\}\stackrel{d}
{=}\{AX(t)+b_r(t), \,  t \in \R^d\}.
$$
Hence for all positive numbers $s\not=r$,
\[
\begin{split}
\{X(s^Et), \,  t \in \R^d\}&\stackrel{d}
{=} \{X(0)+\phi(s^Et), \, t \in \R^d\}\\
&\stackrel{d}{=} \{AX(t)+b_r(t)-
\phi(r^Et)+\phi(s^Et), \, t \in \R^d\}.
\end{split}
\]
Thus $A \in G_s$, which shows $G_r\subset G_s$.
By symmetry, we also have
$G_s\subset G_r$. Therefore $G_r=G_s$ for
all $s\not=r$, and hence $G_r=G$.

(3)$\Rightarrow$(4) This is obvious.

(4)$\Rightarrow$(1) Now we assume (4) holds for some $s<r$.
Let $A\in G_s\cap G_r$. Since
$\{X(s^E t), \,  t \in \R^d\}\stackrel{d}{=}\{AX(t)+b_s(t),
\,  t \in \R^d\}$   and
$\{X(r^Et), \,  t \in \R^d\}\stackrel{d}{=}
\{AX(t)+b_r(t), \,  t \in \R^d\},
$
we obtain that
$\{X(s^Et), \, t \in \R^d\}\stackrel{d}{=}
\{X(r^Et)+\psi(t), \, t \in \R^d\}$ for some function
$\psi: \R^d\to \R^m$. Then
\beqlb\label{s3-l2-1}
\{X((s/r)^Et), \,  t \in \R^d\}\stackrel{d}{=}
\{X(t)+\psi(r^{-E}t), \,  t \in \R^d\}.
 \eeqlb
This shows that $I\in G_{s/r}$.  Let $c_n=(s/r)^n$.
By iterating (\ref{s3-l2-1}) we derive that
$$
 \{X(c_n^Et), \,  t \in \R^d\}\stackrel{d}{=}
 \{X(t)+\psi_n(t), \,  t \in \R^d\},
 $$
where
$\psi_n(t)=\sum\limits_{i=0}^{n-1}\psi(c_i^Er^{-E}t)$.
Hence $I\in G_{c_n}$ for all $n\geq0$. Since $c_n\to 0$
and $I\in$ Aut$(\R^m)$, we arrive at (1). \qed

\begin{lem}\label{s3-l3}
Assume $G\not= G_s$ for some $s>0$. If $A_n\in G_{r_n}$, $A\in
Aut(R^m)$ and $A_n\to A$ as $n\to \infty$, then the sequence $\{r_n\}$
converges to some $r>0$ as $n\to \infty$ and $A\in G_r$.
\end{lem}
\begin{proof}.
Suppose that $\{r_{n_k}\}$ is a subsequence of $\{r_n\}$ and that
$\{r_{n_k}\}$ converges to some $r\in[0, \, \infty]$. Then $0<r< \infty$.
In fact if $r=0$, then $A_{n_k}\to A$ and Lemma
\ref{s3-l2} imply $G=G_s$ for all $s>0$, which is a
contradiction to the assumption. On the other hand,
if $r\to\infty$, then $A^{-1}_{n_k}\in
G_{r_{n_k}^{-1}}\to A^{-1}$ and $r_{n_k}^{-1}\to 0$. By Lemma
\ref{s3-l2}, we also get a contradiction. It follows from
 $$\{X(r_{n_k}^E t), \,  t \in \R^d\}\stackrel{d}
 {=}\{A_{n_k}X(t)+b_{r_{n_k}}(t), \,  t \in \R^d\}$$
and the stochastic continuity of $X$ that
 $$\{X(r^E t), \,  t \in \R^d\}\stackrel{d}
 {=}\{AX(t)+b_r(t), \,  t \in \R^d\}$$
for some function $b_r$. Therefore $A\in G_r$ and
hence from Lemma \ref{s3-l2} we infer that all convergent
subsequences of $\{r_n\}$ have the same limit $r$.
Consequently, $\{r_n\}$ converges to
$r>0$.\qed
\end{proof}

From Lemma \ref{s3-l2} and Lemma \ref{s3-l3}, we
derive the following result.
\begin{cor}\label{s3-c1}
If $G\not= G_s$ for some $s>0$, then $G_1$ is not
a neighborhood of $I$ in $G$.
\end{cor}
\begin{proof}. From Lemma \ref{s3-l2}, the assumption that
$G\not= G_s$ for some $s>0$ implies $G_r\cap G_1=\emptyset$
for all $r\not=1$. Therefore, to prove the corollary, it
is enough to show that there exists a sequence $A_n\in
G_{r_n}$ such that $r_n\not=1$ and $A_n\to I$ as $n\to
 \infty$. This can be proved as follows.

Let $\{r_n\}$ be a sequence with $r_n\not=1$ and $r_n\to1$ as
$n\to \infty$. Take $B_n\in G_{r_n}$. Then by the convergence
of types theorem (see, e.g., \cite[p.55]{S69}), $\{B_n\}$ is
pre-compact in Aut$(\R^m)$. Hence we can find a subsequence
$\{B_{n_k}\}$ such that $B_{n_k}\to B\in$ Aut$(\R^m)$. By
Lemma \ref{s3-l3}, we have $B\in G_1$ and thus $B^{-1}\in G_1$.
Furthermore, by Lemma \ref{s3-l1}, $G_{r_{n_k}}\ni B^{-1}B_{n_k}
\to I\in G_1$. Let $A_k = B^{-1}B_{n_k}$, then the sequence
$\{A_k\}_k$ is what we need.\qed
\end{proof}

Using the above results, we give the proof of Theorem
\ref{s2-t1} as follows.

\medskip

\begin{proof}\ {\it of Theorem \ref{s2-t1}}.\, From Lemma \ref{s3-l2},
we only need to consider two cases.

{\bf Case 1}: $G=G_s$ for all $s>0$.  By Part (2) of Lemma \ref{s3-l2},
we derive that for all constant $c > 0$,
\[
\begin{split}
\big\{X(r^E t),\, t \in \R^d\big\}&\stackrel{d}{=} \{X(0) +\phi(r^E t),\,
t \in \R^d\big\}\\
&\stackrel{d}{=} \{X(t) +\phi(r^E t) -\phi(t),\, t \in \R^d\big\}.
\end{split}
\]
Hence (\ref{s2-t1-1}) holds with $D = {\bf 0}$, which is the matrix
with all entries equal 0, and $b_r(t) = \phi(r^E t) -\phi(t)$.

{\bf Case 2}: $\{G_s, s>0\}$ is a disjoint family. In this case,
$G$ is a closed subgroup of Aut$(\R^m)$. Define $\eta$: $G\to\R$
by $\eta(A)=\ln s$ if $A\in G_s$. It is well-defined and, from
Lemma \ref{s3-l1} and Lemma \ref{s3-l3}, is a continuous
homomorphism between the group $G$ and the group $(\R, +)$. Let
$T(G)$ be the tangent space to $G$ at the identity $I$. It is
well-known that the image of $T(G)$ under the exponential map is a
neighborhood of the identity of $G$; see \cite[p.110.]{C65}.
Therefore, by Corollary \ref{s3-c1}, there exists $A\in T(G)$ such
that $\e^A\not\in G_1$. Furthermore, by the same arguments used in
the proof of Theorem 2.1 of \cite[p.288]{HM82}, we know there is a
$D\in$ End$(\R^m)$ such that $s^D\in G_s$ for every $s>0$. This
implies that \beqlb\label{s3-t1-1} \{X(r^Et), \, t \in
\R^d\}\stackrel{d}{=}\{r^DX(t)+b_r(t), \, t \in \R^d\}. \eeqlb for
some function $b_r(t)$. Note that the linear operators $r^E$ and $r^D$
are continuous on $r\in(0, \, \infty)$. By the convergence of types theorem,
it is not hard to see that $b_r(t)$ is continuous in $(r, t)\in
(0,\, \infty)\times\R^d$. In order to verify the fact $D\in
M(\R^m)$, we let $\{X_0(t), t \in \R^d\}$ be the symmetrization of
$\{X(t), t \in \R^d\}$ and let $\mu(t)$ be the distribution of
$X_0(t)$. Then by (\ref{s3-t1-1})
$$\mu(r^Et)=r^D\mu(t),$$
for all $r>0$ and $t\in\R^d$. Therefore, the characteristic
function of $\mu(t)$, denoted by $\widehat{\mu}_t(z)\, (z\in\R^m)$,
satisfies
\beqlb\label{s3-t1-2}
\widehat{\mu}_{r^Et}(z)=\widehat{\mu}_t(r^{D^*}z)
\eeqlb
for every $r>0$ and $t\in\R^d$, where $D^*$ is the adjoint of $D$.
Suppose $D\not\in M(\R^m)$,  then $D^*\not\in M(R^m)$ either. By
Lemma \ref{s2-l1}, we can find $r_n\to 0$ and $z_0 \in \R^m$ such
that $|r_n^{D^*}z_0|\to \infty$. Let $\alpha_n =|r_n^{D^*}z_0|^{-1}$.
Then by choosing a subsequence if necessary, we have that
$\alpha_nr_n^{D^*}z_0$ converges to some $z_1\in\R^m$ with
$|z_1|=1$. From (\ref{s3-t1-2}), it follows that for all $c \in
\R$
\beqlb\label{s3-t1-3} \widehat{\mu}_{r_n^E t}(c\,\alpha_n
z_0)=\widehat{\mu}_t(c\,\alpha_nr_n^{D^*}z_0).
\eeqlb
Letting $n\to\infty$, since Lemma \ref{s2-l1} implies $r_n^E t
\to 0$,  by the continuity of $\widehat{\mu}_t(\cdot)$, we have that
$\widehat{\mu}_t(cz_1)=\widehat{\mu}_0(0)=1$ for all $c \in \R$. It
follows from Lemma \ref{s2-l2} that $X(t)$ is not full in $\R^m$.
This contradicts the hypothesis that $X$ is proper. Consequently, the
matrix $D$ in (\ref{s3-t1-1}) belongs to $ M(\R^m)$ and the function $b_r(t)$
is continuous in $(0,\, \infty)\times\R^d$.

Now we prove that $X(0)= a$  a.s. for some constant vector  $a \in \R^m$)
if and only if $D\in Q(\R^m)$. From Lemma \ref{s2-l1},
it can be shown that, if $X$ is a stochastically continuous
w.o.s.s. random field and $D\in Q(\R^m)$, then
$X(0)=const$, a.s. Considering the converse assertion,
we note that, in this case, the symmetrization of
$\{X(t), t \in \R^d\}$, i.e. $\{X_0(t), t \in \R^d\}$,
satisfies $X_0(0)=0$ a.s. If $D\not\in Q(\R^m)$, then by Lemma
\ref{s2-l1}, we can find $r_n\to0$ and $z_0$ such that
$|r_n^{D^*}z_0|$ does not converge to $0$. Let
$\alpha_n=|r_n^{D^*}z_0|^{-1}$.  Then choosing a subsequence if
necessary, by the fact  $D\in M(\R^m)$, we have that $\alpha_n$
converges to a finite $\alpha>0$ and that $\alpha_nr_n^{D^*}z_0$
converges to some $z_1\in\R^m$ with $|z_1|=1$. By using
(\ref{s2-t1-1}) and the same argument as that leads to
(\ref{s3-t1-2}) and (\ref{s3-t1-3}) we derive
\beqlb\label{s3-t1-4}
\widehat{\mu}_{r_n^E t}(c\, \alpha_n
z_0)=\widehat{\mu}_t(c\, \alpha_n r_n^{D^*}z_0)
\eeqlb
for all $c \in \R$. Letting $n\to\infty$, we have that
$\widehat{\mu}_t(c z_1)=\widehat{\mu}_0(c \alpha z_0)=1.$
Then by Lemma \ref{s2-l2}, $X(t)$ is not full in $\R^m$.
This contradiction implies $D\in Q(\R^m)$.

The last assertion follows from the stochastic continuity of $X$
and (\ref{s2-t1-1}). This finishes the proof of Theorem \ref{s2-t1}.\qed
\end{proof}

\medskip
\begin{proof}\ {\it of Corollary \ref{s2-c2}}.
For every $t \in \R^d \backslash \{0\}$ we use polar coordinate decomposition
under the operator $E$ to write it as $t = \tau_E(t)^E l(t)$.
We define $b(t) = b_{\tau_E(t)}\big(l(t)\big)$ for $t \in \R^d \backslash \{0\}$.
Then from (\ref{Eq:Con3}) we derive that for all $r > 0$ and
$t \in \R^d \backslash \{0\}$,
\[
b_{r\tau_E(t)}\big(l(t)\big) = b_r\big(\tau_E(t)^E l(t)\big) + r^D b_{\tau_E(t)}
\big(l(t)\big),
\]
which can be rewritten as
\[
b\big((r\tau_E(t))^El(t)\big) = b_r\big(t\big) + r^D b \big(t\big).
\]
This implies  $b_r(t) = b\big(r^E t\big) -r^D b(t)$ for all $r > 0$ and $t \in \R^d
\backslash \{0\}$. In the case when $X(0) = a$ a.s., (\ref{s2-t1-1}) implies
$b_r(0) = a - r^D a$, which shows that (\ref{Eq:br}) still holds for $t = 0$.
\qed
\end{proof}
\medskip

\begin{proof}\ {\it of Theorem \ref{s2-t2}}.\,
The proof is similar to the proof of Theorem \ref{s2-t1}, with some minor modifications.
For proving Part (i), we define $G_r$ to be the set of $A\in \text{Aut}(\R^m)$ such that
$\{X(r^Et),\, t \in \R^d\}\stackrel{d}{=}\{AX(t),
\, t \in \R^d\};$ and for proving Part (ii), we define $G_r$ to be the set
of $A\in \text{Aut}(\R^m)$ such that
$\{X(r^Et),\, t \in \R^d\}\stackrel{d}{=}\{AX(t)+ b(r),
\, t \in \R^d\},$ for some function $b: (0,\, \infty) \to \R^m$.
The rest of the proof follow similar lines as in the proof of Theorem \ref{s2-t1}
and is omitted.
\qed
\end{proof}

\medskip
We end this section with two more propositions. Proposition \ref{s2-p1}
shows that, if a $(d, m)$-random field $X$ is w.o.s.s. with time-variable
scaling exponent $E$, then along each direction of the eigenvectors of $E$,
$X$ is an ordinary one-parameter operator-self-similar process as defined
by Sato \cite{S91}. Proposition \ref{s3-p1}
discusses the relationship between w.o.s.s. random fields and o.s.s. random
fields in the sense of Hudson and Mason (see (ii) in Remark 1.1).

\begin{prop}\label{s2-p1}
Let $X=\{X(t), t\in\R^d\}$ be a stochastically continuous
and proper $(E, D)$-w.o.s.s. random field with values in
$\R^m$. Let $\lambda$ be a positive eigenvalue of $E$ and
$\xi \in\R^d$ satisfy $E \xi=\lambda \xi$. Denote
$\widetilde{b}_r(u)= b_r(u\xi)$ for all $u\in \R$. Then the
following statements hold.
\begin{itemize}
\item[(i)]\, There exists a
continuous function $f(u)$ from $\R\backslash\{0\}$ to $\R^m$, such
that $\widetilde{b}_r(u)=f(u r^\lambda)-r^Df(u)$ for all $u\not=0$
and $r>0$.
\item[(ii)]\, If $D\in Q(\R^m)$, then
$f(u)$ can be defined at $u=0$ such that $f(u)$ is continuous in
$\R$. Moreover, the stochastic process $Y= \{Y(u),\, u \in \R\}$
defined by $Y(u)=X(u\xi)-f(u)$ satisfies that for any $r>0$
$$
\big\{Y(r u),\, u\in\R\big\}\stackrel{d}{=}\big\{r^{D/\lambda}Y(u),
\, u \in\R \big\}.
$$
\end{itemize}
\end{prop}
\begin{proof}. By Corollary \ref{s2-c1}, we have that
$$
b_{r_1r_2}(u\xi)=b_{r_1}(r_2^Eu\xi)+r_1^Db_{r_2}(u\xi)
$$
for all $r_1, r_2 > 0$. Since $E\xi=\lambda \xi$ and
$r_2^Eu\xi=ur_2^\lambda \xi$, we have
\beqlb\label{s2-p1-1}
 b_{r_1r_2}(u\xi)=b_{r_1}(r_2^\lambda u\xi)+r_1^Db_{r_2}(u\xi).
\eeqlb
Define $f(u)=b_{u^{1/\lambda}}(\xi)$ for $u>0$ and
$f(u)=b_{|u|^{1/\lambda}}(-\xi)$ for $u<0$. Then the continuity
of $f(u)$ on $\R\backslash \{0\}$ follows from the continuity
of $b_r(t)$. Moreover, from (\ref{s2-p1-1}) it follows that
\beqlb
 \widetilde{b}_{r_1}(r_2^\lambda)&=&-r_1^D f(r_2^\lambda)
 +f(r_1^\lambda  r_2^\lambda),\label{s2-p1-2}
 \\
 \widetilde{b}_{r_1}(-r_2^\lambda)&=&-r_1^D f(-r_2^\lambda)
 +f(-r_1^\lambda  r_2^\lambda).\label{s2-p1-3}
\eeqlb
Writing $u = r_2^{\lambda}$ or $- r_2^{\lambda}$ and $r = r_1$,
we see that (\ref{s2-p1-2}) and (\ref{s2-p1-3}) yield that
\beqlb\label{s2-p1-4}
\widetilde{b}_r(u)=f(u r^\lambda)-r^Df(u)
\eeqlb
for all $r>0$, $u\not=0$. This proves (i).

Suppose $D\in Q(\R^m)$. Lemma \ref{s2-l1} implies that $r^DX(\xi)\to
0$ and $r^DX(-\xi)\to 0$ in probability as $r\to 0$. Theorem
\ref{s2-t1} and the convergence of types theorem indicate that,
as $r \to 0+$, the limits of $b_r(\xi)$ and $b_r(-\xi)$ exist and
coincide. Hence, we can
define $f(0):=\lim_{r\to0}b_r(\xi)$. Then $f(u)$ is continuous in
$\R$. Combining (\ref{s2-t1-1}) and (\ref{s2-p1-4}) yields that
for all $r>0$, $u\in\R$,
\[
\begin{split}
\big\{X(r^\lambda u\xi),\, u \in \R\big\}
& =\{X(r^Eu\xi),\, u \in \R\}\\
&\stackrel{d} {=} \big\{r^DX(u \xi)+
f(ur^{\lambda})-r^Df(u), \, u \in \R \big\}.
\end{split}
\]
Hence for the process $Y= \{Y(u), u \in \R\}$ defined by
$Y(u)=X(u \xi)-f(u)$, we have  $\{Y(r^\lambda u), u \in \R\}
\stackrel{d} {=} \{r^D Y(u), u \in \R\}$.
Equivalently, $Y$ is $D/\lambda$-o.s.s. This finishes the proof.
\qed
\end{proof}

\begin{prop}\label{s3-p1}
Let $X=\{X(t), t\in\R^d\}$ be a stochastically continuous
and proper $(E,D)$-w.o.s.s. random field with values in $\R^m$.
Suppose $E$ has two different positive eigenvalues $\lambda_1$ and
$\lambda_2$. Then $X$ is o.s.s. in the sense of Hudson and Mason if
and only if $b_r(t)$ in (\ref{s2-t1-1}) only depends on $r$ and $|t|$
for all $r>0$ and $t\in\R^d$.
\end{prop}

\begin{proof}. The ``necessity " part is obvious, because,
for every $(E,D)$-o.s.s. random field in the sense of Hudson and Mason,
the function $b(r)$ does not depend on $t$.

In the following, we prove the sufficiency. Suppose
$b_r(t)$ only depends on $r$ and $|t|$ for all $r>0$
and $t\in\R^d$. Then we can
find a function $g$ on $\R^2$ such that $b_r(t)=g(r, |t|)$.  By
Corollary \ref{s2-c1}, we have that for all $r_1, r_2>0$ and
$t\in\R^d$
 \beqlb\label{s3-p1-1}
 g(r_1r_2, |t|)=g(r_1, |r_2^Et|)+r_1^Dg(r_2, |t|).
 \eeqlb
Let $\xi_1$, $\xi_2$ be the eigenvectors of $E$ corresponding to
$\lambda_1$ and $\lambda_2$, respectively. Without loss of
generality, we assume $|\xi_1|=|\xi_2|=1$ and $\lambda_2<\lambda_1$.
Then from (\ref{s3-p1-1}), we have that
 \beqnn
 g(r_1r_2, 1)&=&g(r_1, r_2^{\lambda_1})+r_1^Dg(r_2, 1),
 \\ g(r_1r_2, 1)&=&g(r_1, r_2^{\lambda_2})+r_1^Dg(r_2, 1),
 \eeqnn
where we have used the facts $r^E\xi_1=r^{\lambda_1}\xi_1$ and
$r^E\xi_2=r^{\lambda_2}\xi_2$. Therefore, we derive that
$g(r, u^{\lambda_1})=g(r, u^{\lambda_2})$for any
$r>0$ and $u\geq 0$ and hence, for all $n\geq  1$,
\beqlb\label{s3-p1-2}
g(r, u)=g(r, u^{\lambda_2/\lambda_1})=g(r,
u^{\lambda_2^n/\lambda_1^n}).
\eeqlb
Note that by Theorem \ref{s2-t1}, $g(r, u)$ is continuous
on $(0, \, \infty)\times[0, \, \infty)$. Therefore,
 \beqlb\label{s3-p1-3}
 g(r, 0)=\lim_{u\to 0}g(r, u)
 \eeqlb
and for any $u>0$, letting $n\to \infty$, from (\ref{s3-p1-2})
we get that
\beqlb\label{s3-p1-4}
  g(r, u)=g(r, 1).
\eeqlb
Combining (\ref{s3-p1-3}) with (\ref{s3-p1-4}), we obtain that
$g(r, 0)=g(r, 1)$ and hence for all $r>0$ and $u\geq0$, $g(r, u)
=g(r, 1)$. This means $b_r(t)=g(r, 1)$ is independent of $t$.
Hence the random field $X$ is o.s.s. in the sense of Hudson and Mason.
\qed
\end{proof}

\bigskip

\section{Construction of o.s.s. stable random fields:
Proofs of Theorems \ref{s4-t1}--\ref{s5-t2}}

This section is concerned with constructing $(E,D)$-o.s.s. random
fields by using stochastic integrals with respect to S$\alpha$S
random measures. In particular, we prove the remaining theorems in
Section 2.

Note that Theorem \ref{s4-t1} is a multiparameter extension of
Theorem 4.1 in \cite{MM94} and can be proved by using essentially
the same argument with some modifications. Hence the proof of
Theorem \ref{s4-t1} is omitted here. In the following, we first
prove Theorem \ref{s4-t2}.

\begin{proof}\ {\it of Theorem \ref{s4-t2}}. We divide the proof
into two steps.

(1) When $\widetilde{Q}(u)$ is a simple function of the form
\beqlb\label{s4-t2-2}
\widetilde{Q}(u)=\widetilde{Q}_1(u)+i\widetilde{Q}_2(u)=
\sum_{j=1}^k R_j{\bf 1}_{A_j}(u)
+i\sum_{j=1}^k I_j{\bf 1}_{A_j}(u),
\eeqlb
where $R_j, I_j\in {\rm End}(\R^m)$ and $A_j, j=1,2,\cdots, k$
are pairwise disjoint sets in ${\mathcal M}$, we define
$$
\widetilde I(\widetilde{Q})=\sum_{j=1}^k \Big(R_j
\widetilde{M}_R(A_j)-I_j\widetilde{M}_I(A_j)\Big).
$$
Then for any $\theta \in \R^m$, from (\ref{s4-d1-2}), we obtain
that
\beqlb\label{s4-t2-3}
\E\Big[\e^{i \l\theta, \widetilde I(\widetilde{Q})\r}\Big]
&=&\exp\Bigg\{-\sum_{j=1}^{k}\Big(\big|R_j^*\theta\big|^2
+\big|I_j^*\theta\big|^2\Big)^{\alpha/2}\lambda(A_j)\Bigg\}\nonumber \\
&=&\exp\Bigg\{-\int_{\R^d}\Big(\big|\widetilde{Q}_1(u)^*\theta\big|^2
+\big|\widetilde{Q}_2(u)^*\theta\big|^2\Big)^{\alpha/2}\, \d
u\Bigg\}.
\eeqlb

(2). When $\{\widetilde{Q}(u)\}$ fulfills $\int_{\R^d}
\big(\|\widetilde{Q}_1(u)\|^\alpha  +\|\widetilde{Q}_2(u)\|^\alpha
\big) \d u<\infty$, we can choose a sequence of simple functions
$\{\widetilde{Q}^{(n)}(u)=\widetilde{Q}_1^{(n)}(u)
 +i \widetilde{Q}_2^{(n)}(u)\}$ of the form (\ref{s4-t2-2})
 such that as $n\to\infty$,
 \beqlb\label{s4-t2-4}
 \int_{\R^d}\big\|\widetilde{Q}_1(u)^*-\widetilde{Q}_1^{(n)}(u)^*
 \big\|^\alpha\d  u\to 0
 \eeqlb
 and
 \beqlb\label{s4-t2-5}
 \int_{\R^d}\big\|\widetilde{Q}_2(u)^*-\widetilde{Q}_2^{(n)}(u)^*
 \big\|^\alpha\d u \to 0.
 \eeqlb
By the linearity of $\widetilde I(\cdot)$ we have
$$
\widetilde I(\widetilde{Q}^{(n)})- \widetilde I(\widetilde{Q}^{(\ell)})=
\widetilde I(\widetilde{Q}^{(n)}
-\widetilde{Q}^{(\ell)}),
$$
and $\E(\e^{i \l\theta, \widetilde I(\widetilde{Q}^{(n)}-\widetilde{Q}^{(\ell)})\r})$
equals
\beqnn
&&\exp\Bigg\{-\int_{\R^d}\Big(\big|\big(\widetilde{Q}_1^{(n)}(u)^*-
  \widetilde{Q}_1^{(\ell)}(u)^*\big)\theta\big|^2+
  \big|\big(\widetilde{Q}_2^{(n)}(u)^*- \widetilde{Q}_2^{(\ell)}(u)^*\big)
  \theta\big|^2\Big)^{\alpha/2}\d   u\Bigg\}
  \\&&\geq \exp\Bigg\{-\int_{\R^d}\big|\big(\widetilde{Q}_1^{(n)}(u)^*
  -\widetilde{Q}_1^{(\ell)}(u)^*\big)\theta\big|^\alpha\d u-\int_{\R^d}
  \big|\big(\widetilde{Q}_2^{(n)}(u)^*-\widetilde{Q}_2^{(\ell)}(u)^*\big)
\theta\big|^\alpha\d u\Bigg\}
\eeqnn
 which converges to $1$ as $\ell, n\to\infty$ by (\ref{s4-t2-4}) and
 (\ref{s4-t2-5}). Thus $\widetilde I(\widetilde{Q}^{(n)})-
 \widetilde I(\widetilde{Q}^{(\ell)})\to 0$
 in probability as $\ell, n\to\infty$, and $ \widetilde I(\widetilde{Q}^{(n)})$
 converges to an $\R^m$-valued random vector in probability. It is
 easy to see that the limit does not depend on the choice of
 $\{\widetilde{Q}^{(n)}\}$. Therefore, we can define $\widetilde I (\widetilde{Q})$ as
 the limit of $\widetilde I(\widetilde{Q}^{(n)})$, and hence
 \beqnn
\E\Big(\e^{i \l \theta, \widetilde I(\widetilde{Q})\r}\Big)&=&
\lim_{n\to\infty}\E\Big(\e^{i\l\theta,
\widetilde I(\widetilde{Q}^{(n)})\r}\Big)\\
&=&\exp\bigg\{-\int_{\R^d}\Big(\sqrt{\big|\widetilde{Q}_1(u)^*\theta\big|^2
 +\big|\widetilde{Q}_2(u)^*\theta\big|^2}\,\Big)^\alpha \d
 u\bigg\}.
 \eeqnn
The proof of Theorem \ref{s4-t2} is completed. \qed
\end{proof}

In order to prove Theorem \ref{s5-t1} and Theorem \ref{s5-t2}, we
will use the following change of variable formula from \cite{BMS07}.

\begin{lem}\label{s5-l1}
\upshape{(\cite[Proposition 2.3]{BMS07})} Let $E \in Q(\R^d)$
be fixed and let $(\tau(x), l(x))$ be the polar coordinates of $x$
under the operator $E$. Denote $\Sigma_0:=\{\tau(x)=1\}$. Then
there exists a unique finite Radon measure $\sigma$ on $\Sigma_0$
such that for all $f\in L^1(\R^d, \d x)$,
$$
\int_{\R^d}f(x)\, \d x=\int_0^\infty\int_{\Sigma_0}
f(r^E\theta)\sigma(\d \theta)r^{q-1}\d r.
$$
\end{lem}

We also need the following lemma which is due to Maejima and Mason
\cite{MM94}. For more precise estimates on $\|r^D\|$ see Mason
and Xiao \cite{MX00}.
\begin{lem}\label{s5-l2}
Let $D\in Q(\R^m)$ and let $h>0$ and $H>0$ be the minimal and
maximal real parts of the eigenvalues of $D$, respectively. Then
for any $\delta>0$, there exist positive constants $C_3$ and $C_4$
such that
 $$
 \|r^D\|\leq\begin{cases}
 C_3\,r^{h-\delta}, \quad &\hbox{ if } 0<r\leq 1,
 \\ C_4\, r^{H+\delta}, & \hbox{ if } r>1.
 \end{cases}
 $$
\end{lem}

Now we are in position to prove Theorem \ref{s5-t1}.

\begin{proof}\ {\it of Theorem \ref{s5-t1}}. We divide the proof
into four  parts.

(i). First we show that the stochastic integral in (\ref{Eq:Xphi})
is well defined. By Theorem \ref{s4-t1}, it
suffices to show that for all $x \in \R^d$
\beqlb\label{s5-t1-1}
\Upsilon_\phi^\alpha(x)=\int_{\R^d}
\Big\|\phi(x-y)^{D-qI/\alpha}-\phi(-y)^{D-qI/\alpha}\Big\|^\alpha \d
y<\infty.
\eeqlb
Let $(\tau(x), l(x))$ be the polar coordinates of $x$ under
operator $E$. By the fact that $\phi$ is $E$-homogeneous, we see
that
$$
 \phi(y)=\tau(y)\phi(l(y))\qquad \forall \ y\in\R^d.
$$
Then by $(\ref{s5-2})$, we have that
 \beqlb\label{s5-3}
 m_\phi\tau(y)\leq\phi(y)\leq M_\phi \tau(y).
 \eeqlb
Therefore, there exists a constant $C_5>0$ such that
$$
\big\|\phi(y)^{D-qI/\alpha}\big\|^\alpha\leq
C_5\,\big\|\tau(y)^{D-qI/\alpha}\big\|^\alpha.
$$
Note that
$$
M_1=\sup_{m_\phi\leq r\leq
M_\phi}\|r^E\|>0\;\;\text{ and }\;\;M_2=\sup_{1/M_\phi\leq r\leq
1/m_\phi}\|r^E\|>0
$$
are finite because $r^E$ is continuous in $r$
and $\|r^E\|\not=0$ for all $r>0$, and that
 $$0<m=\inf_{y\in\Sigma_0}|y|\leq M=\sup_{y\in\Sigma_0}|y|<\infty,$$
since $\Sigma_0$ is compact and $0\not\in\Sigma_0$. Therefore, from
 $$\phi^{-E}(y)y=\phi^{-E}(y)\tau(y)^El(y)=(\phi^{-1}(y)\tau(y))^El(y)$$
and (\ref{s5-3}), it follows that
 \beqlb\label{s5-4}
 0<\frac{m}{M_1}\leq\big|\phi^{-E}(y)y\big|\leq MM_2<\infty.
 \eeqlb
Since $\phi$ is $(\beta, E)$-admissible, for any $z$ with
$\frac{m}{M_1}\leq |z|\leq MM_2$ there exists a positive constant
$C_1>0$ such that \beqlb\label{s5-5}
 |\phi(x+z)-\phi(z)|\leq C_1\tau(x)^\beta
\eeqlb
for all $x \in \R^d$ with $\tau(x)\leq 1$. For any $\gamma>0$, on the set
$\{y \in \R^d:\, \tau(y)\leq \gamma\}$, we have
$$
\Big\|\phi(x-y)^{D-qI/\alpha}-\phi(-y)^{D-qI/\alpha}\Big\|^\alpha
\leq 2 \Big\|\phi(x-y)^{D-qI/\alpha}\Big\|^\alpha
+2\Big\|\phi(-y)^{D-qI/\alpha}\Big\|^\alpha.
$$
Consequently, by Lemma \ref{s5-l1}, Lemma \ref{s5-l2} and the fact
$\tau(-y)=\tau(y)$, there exist constants $C_6>0$ and
$0<\delta<\alpha h$ such that,
  \beqnn
  \int_{\tau(y)\leq\gamma}\Big\|\phi(-y)^{D-qI/\alpha}\Big\|^\alpha\d y
  &\leq&\int_{\tau(y)\leq\gamma} C_5\Big\|\tau(y)^{D-qI/\alpha}\Big\|^\alpha\d y
  \\&\leq& C_6\int_{\tau(y)\leq\gamma} \tau(y)^{\alpha
  h-q-\delta}\d y<\infty.
  \eeqnn
At the same time, (\ref{s5-1}) implies
\[\begin{split}
\big\{y\in \R^d: \tau(x+y)\leq
\gamma\big\}&\subset\big\{y \in \R^d:\tau(y)\leq C_0(\gamma+\tau(-x))\big\}\\
&=\big\{y \in \R^d: \tau(y)\leq C_0\,(\gamma+\tau(x))\big\}.
\end{split}
\]
Consequently we derive that
\beqnn
&&\int_{\tau(y)\leq\gamma}\Big\|\phi(x-y)^{D-qI/\alpha}\Big\|^\alpha\d
y=\int_{\tau(x+y)\leq\gamma}\Big\|\phi(-y)^{D-qI/\alpha}\Big\|^\alpha\d
y\\
&&\qquad\leq C_6\int_{\tau(y)\leq C_0(\gamma+\tau(x))}
\tau(y)^{\alpha   h-q-\delta}\, \d y<\infty. \eeqnn Combining the
above shows that for any $\gamma>0$ \beqlb\label{s5-t1-2}
\int_{\tau(y)\leq\gamma}
\Big\|\phi(x-y)^{D-qI/\alpha}-\phi(-y)^{D-qI/\alpha}\Big\|^\alpha
\d y<\infty. \eeqlb

Next we consider the integral on the set $\{y \in \R^d: \tau(y)>\gamma\}$
for sufficiently large $\gamma$ such that $\phi(-y)^{-1}\tau(x)<1$,
$C_1\phi(-y)^{-\beta}\tau(x)^\beta<1/2$
and $\phi(-y)>1$. This is possible because of (\ref{s5-3}). Note that for
any $3/2>u>1/2$, from the fact
$$
\frac{\d s^{D-qI/\alpha}}{\d s}=\frac{\d}{\d s}\e^{\ln s (D-qI/\alpha)}
=(D-qI/\alpha)s^{\alpha  D-(1+q/\alpha)I}
$$
and Lemma \ref{s5-l2}, there exists $C_7>0$ such that
\beqlb \label{s5-t1-3a}
\begin{split}
\big\|u^{D-qI/\alpha}-I\big\| &\leq \Big\|D-\frac{q I}{\alpha}
\Big\|\int_{1\wedge u}^{1\vee u} \big\|s^{D-(1+q/\alpha)
I}\big\|\d s\\
&\leq C_7\,\Big\|D-\frac{q I}{\alpha}\Big\|\cdot|u-1|.
\end{split}
\eeqlb
Since $\phi$ is $E$-homogenous and $\phi(-y) > 0$, we have
\beqlb\label{s5-t1-3}
&&\Big\|\phi(x-y)^{D-qI/\alpha}-\phi(-y)^{D-qI/\alpha}\Big\|\nonumber
\\&&\qquad
\le \Big\|\phi(-y)^{D-qI/\alpha}\Big\|\cdot
\Big\|\phi(\phi^{-E}(-y)x-\phi^{-E}(-y)y)^{D-qI/\alpha}-I\Big\|.
\eeqlb
On the other hand, $\tau(\phi^{-E}(-y)x)=\phi^{-1}(-y)\tau(x)<1$ and
$\phi(-\phi^{-E}(-y)y)=1$, we can use (\ref{s5-4}) and
(\ref{s5-5}) to derive
\begin{equation}\label{s5-t1-3c}
\begin{split}
\Big|\phi(\phi^{-E}(-y)x-\phi^{-E}(-y)y)-1\Big|
&\le C_1\, \big[\tau(\phi^{-E}(-y)x)\big]^\beta \\
&=C_1\phi^{-\beta}(-y)\tau(x)^\beta.
\end{split}
\end{equation}
Since the last term is less than $1/2$, we can apply (\ref{s5-t1-3a})
with $u = \phi(\phi^{-E}(-y)x-\phi^{-E}(-y)y)$. Hence, we derive from
(\ref{s5-t1-3}), (\ref{s5-t1-3a}), (\ref{s5-t1-3c}) and Lemma \ref{s5-l2}
that for some $0<\delta_1<(\beta-H)\alpha$
\beqnn
&&\Big\|\phi(x-y)^{D-qI/\alpha}-\phi(-y)^{D-qI/\alpha}\Big\|^\alpha
\nonumber \\
&&\qquad\leq C_7^\alpha\, \Big\|\phi(-y)^{D-qI/\alpha}\Big\|^\alpha\cdot
\Big\|D-\frac{q}{\alpha}
I\Big\|^\alpha\, \big|\phi(\phi^{-E}(-y)x-\phi^{-E}(-y)y)-1\big|^\alpha
\nonumber \\&&\qquad\leq C_8
\, \Big\|\phi(-y)^{D-qI/\alpha}\Big\|^\alpha\cdot
\phi(-y)^{-\alpha\beta}\tau(x)^{\alpha\beta}\\
&&\qquad\leq C_9\, \phi(-y)^{\alpha H+\delta_1-q-\alpha\beta}\,
\tau(x)^{\alpha\beta} \\
&&\qquad\leq C_{10}\, \tau(y)^{\alpha H+\delta_1-q-\alpha\beta}\,
\tau(x)^{\alpha\beta}.
\eeqnn
This and Lemma \ref{s5-l1} yield
\beqlb\label{s5-t1-4}
\begin{split}
&\int_{\tau(y)>\gamma}\Big\|\phi(x-y)^{D-qI/\alpha}-
\phi(-y)^{D-qI/\alpha}\Big\|^\alpha\d y \\
&\qquad \le C_{11}\,\tau(x)^{\alpha\beta}
\int_\gamma^\infty r^{-\alpha (\beta-H) +\delta_1 -1}\,
dr<\infty.
\end{split}
\eeqlb
Combining (\ref{s5-t1-2}) and (\ref{s5-t1-4}), we get
(\ref{s5-t1-1}) which shows that $X_\phi$ is well defined.

(ii). To show the stochastic continuity of the $\alpha$-stable
random field $X_\phi$, it is sufficient to verify that
${\E}(\exp\{i\langle\theta,
X_{\phi}(x+x_0)-X_{\phi}(x_0)\rangle\}) \to 1 $ for all $x_0\in
\R^d$ and $\theta\in{\mathrm R}^m$. By Theorem 2.2 it is enough to
prove that for every $x_0\in \R^d$, we have \beqlb\label{s5-t1-5}
\int_{\R^d}
\Big\|\phi(x_0+x-y)^{D-qI/\alpha}-\phi(x_0-y)^{D-qI/\alpha}
\Big\|^\alpha\d y\to0 \qquad \text{as}\;\;x\to 0.
 \eeqlb
By a change of variables, (\ref{s5-t1-5}) holds if
$$
 \Upsilon_{\phi}^\alpha(x)\to 0\quad\text{as}\;\; x\to 0.
$$
From the continuity of $\phi$ and the continuity of the function
$(r, D)\to r^D$ (see Meerschaert and Scheffler
\cite[p.30]{MSBook}), we have  that
$$
\Big\|\phi(x-y)^{D-qI/\alpha}-\phi(-y)^{D-qI/\alpha}\Big\|\to 0
\quad\text{as}\;\; x\to0
$$
for every $y\in\R^d\backslash \{0\}$. Moreover, from the argument
in Part (i), it follows that for a sufficiently large $\gamma>0$,
there exists a positive constant $C_{12}$ such that
$\big\|\phi(x-y)^{D-qI/\alpha}-\phi(-y)^{D-qI/\alpha}\big\|^\alpha$
is bounded by
 \beqnn
\Phi(x,y)&=&C_{12}{\bf 1}_{\{\tau(y)\leq\gamma\}} (\tau(y)^{\alpha
h-\delta-q}+\tau(x-y)^{\alpha h-\delta-q})\\
&&+C_{12} \tau(y)^{\alpha H+\delta_1-qI-\alpha\beta}\tau(x)^{\alpha\beta}
  {\bf 1}_{\{\tau(y)>\gamma\}},
  \eeqnn
where $0<\delta<\alpha h$ and $0<\delta_1<\alpha(\beta-H)$. It is
easy to see that  $\Phi(x, y)\to \Phi(y)$ a.e. as $x\to 0$, where
$$
\Phi(y)=2C_{12}{\bf 1}_{\{\tau(y)\leq\gamma\}}\tau(y)^{\alpha
h-\delta-q},
$$
 and that
$$
\int_{\R^d}\Phi(x, y)\d y\to\int_{\R^d}\Phi(y)\d y.
$$
By the generalized dominated convergence theorem (see
\cite[p.492]{E86}), (\ref{s5-t1-5}) holds.

(iii). In order to show that for all $r > 0$
$$
\big\{X_\phi(r^Ex),\, x \in
\R^d\big\}\stackrel{d}{=}\big\{r^DX_\phi(x),\, x \in \R^d\big\},
$$
we note that, by Theorem \ref{s4-t1}, it is sufficient to
prove that for all $k\geq 1$, $x_j \in \R^d$ and $\theta_j\in\R^m$
($j=1,2,\cdots, k$)
\begin{equation}\label{Eq:415}
\begin{split}
&\int_{\R^d}\bigg | \sum_{j=1}^k \Big(\phi(r^Ex_j-y)^{D-qI/\alpha}-
\phi(-y)^{D-qI/\alpha}\Big)^* \theta_j\bigg|^\alpha\d y\\
& = \int_{\R^d}\bigg | \sum_{j=1}^k
r^D\Big(\phi(x_j-y)^{D-qI/\alpha}- \phi(-y)^{D-qI/\alpha}\Big)^*
\theta_j\bigg|^\alpha\d y.
\end{split}
\end{equation}
This can be verified by an appropriate change of variables. By the
$E$-homogeneity of $\phi$, we have
\beqnn
&&
\int_{\R^d}\bigg|\sum_{j=1}^k \Big(\phi(r^Ex_j-y)^{D-qI/\alpha}-
\phi(-y)^{D-qI/\alpha}\Big)^* \theta_j\bigg|^\alpha\,\d y\nonumber
\\
&&\qquad=\int_{\R^d}\bigg|\sum_{j=1}^k
\bigg( r^{D-qI/\alpha}\Big(\phi(x_j-r^{-E}y)^{D-qI/\alpha}
-\phi(-r^{-E}y)^{D-qI/\alpha}\Big)\bigg)^*\theta_j \bigg|^\alpha\d y\\
&&\qquad =\int_{\R^d}\bigg|\sum_{j=1}^k \bigg(r^{D-qI/\alpha}\Big
(\phi(x_j-y)^{D-qI/\alpha}-\phi(-y)^{D-qI/\alpha}\Big)\bigg)^*\theta_j\bigg|^\alpha
r^q\d y\\
&&\qquad =\int_{\R^d}\bigg|\sum_{j=1}^k \bigg(r^D\Big(\phi(x_j-y)^{D-qI/\alpha}
-\phi(-y)^{D-qI/\alpha}\Big)\bigg)^*\theta_j\bigg|^\alpha\d y.
\eeqnn
This proves (\ref{Eq:415}) and thus $X$ is an $(E, D)$-o.s.s. random field.

(iv). In the same way, we can verify that $X_\phi(x)$ has stationary
increments. The details are omitted.\qed
\end{proof}

Finally, we prove Theorem \ref{s5-t2}.

\begin{proof} {\it of Theorem \ref{s5-t2}}.
The proof is essentially an extension of the proofs of Theorem 4.1
and Corollary 4.2 in \cite{BMS07}. We only show that the stable
random field $\widetilde{X}_\psi$ is well defined. Then properness of
the $\widetilde{X}_\psi$ follows from the fact that for the matrix
$\psi(y)^{-D-qI/\alpha}$ is invertible for every $y \in \R^d\backslash\{0\}$.
The verification of the rest conclusions on $\widetilde{X}_\psi$
is left to the reader.

By Theorem \ref{s4-t2}, it suffices to show that
$$
\Upsilon_{\psi}(x):=\int_{\R^d}\Big(|1-\cos\langle x,
y\rangle|^\alpha+
|\sin\langle x, y\rangle|^\alpha\Big)\Big\|\psi(y)^{-D-q I/\alpha}
\Big\|^\alpha \,\d y<\infty.
$$
Let $(\tau_1(x), l_1(x))$ be the polar coordinates of $x$
under the operator $E^*$. By (\ref{s5-2}) and Lemma
\ref{s5-l2}, there exist
$0<\delta< [\frac{\alpha}{1+\alpha}(a_1-H)]\wedge
(\alpha h)$ and $C_{13}>0$, such that
 $${\bf 1}_{\{\tau_1(y)\geq 1\}}\Big\|\psi(y)^{- D-q
 I/\alpha}\Big\|^\alpha
 \leq C_{13} \tau_1(y)^{-\alpha h+\delta-q},
 $$
and
 $$
{\bf 1}_{\{\tau_1(y)<1\}}\Big\|\psi(y)^{-D-q
I/\alpha}\Big\|^\alpha\leq C_{13} \tau_1(y)^{-\alpha H-\delta-q}.
$$
Then by Lemma \ref{s5-l1}, $\Upsilon_{\psi}(x)$ is bounded by
\beqnn
&& C_{13}\int_1^\infty\int_{\Sigma_0}\Big(|1-\cos\langle x,
r^{E^*}\theta\rangle|^\alpha
 +|\sin\langle x, r^{E^*}\theta\rangle |^\alpha\Big)\,
 r^{-\alpha h+\delta-1}\sigma(\d\theta)\d r
 \\&&+C_{13}\int_0^1\int_{\Sigma_0}\Big(|1-\cos\langle x,
 r^{E^*}\theta\rangle |^\alpha
+|\sin\langle x, r^{E^*}\theta\rangle |^\alpha\Big)\, r^{-\alpha
H-\delta-1}\sigma(\d\theta)\d r.
\eeqnn
Note that there is a constant $C_{14}>0$ such that
$$
\big|1-\cos\langle x, r^{E^*}\theta\rangle \big|^\alpha+
\big|\sin\langle x, r^{E^*}\theta\rangle \big|^\alpha \leq
C_{14}\big(1+|x|^\alpha\big) \big(r^{\alpha(a_1-\delta)}\wedge 1\big).
$$
Therefore
\beqnn
\Upsilon_{\psi}(x)\leq C_{15}\, \big(1+|x|^\alpha\big)\sigma(\Sigma_0)
\Bigg[\int_1^\infty r^{-\alpha h+\delta-1}\d
r+\int_0^1 r^{\alpha(a_1-H)-(1+\alpha)\delta-1}\d r\Bigg].
\eeqnn
Since $0<\delta< \big[\frac{\alpha}{1+\alpha}(a_1-H)\big]\wedge
(\alpha h)$ and $\sigma$ is a finite measure on $\Sigma_0$, we have
$\Upsilon_{\psi}(x)<\infty$ for every $x\in\R^d$. This proves that
$\widetilde{X}_\psi$ is a well-defined stable random field. \qed
\end{proof}

The moving-average-type and harmonizable-type
o.s.s. stable random fields are quite different
(e.g., even in the special case of $D=I$, the
regularity properties of ${X}_{\phi}$ and
$\widetilde{X}_{\psi}$ are different.) From both
theoretical and applied points of view, it is important
to investigate the sample path regularity and fractal properties
of the $(E, D)$-o.s.s. $S \alpha S$-random fields $X_\phi$ and
$\widetilde{X}_\psi$. We believe that many
sample path properties such as H\"older continuity and fractal
dimensions of $X_\phi$ and $\widetilde{X}_\psi$ are
determined mostly by the real parts of the eigenvalues of $E$ and
$D$. It would be interesting to find out the precise connections.
We refer to Mason and Xiao \cite{MX00}, Bierm\'e and Lacaux
\cite{BL07} and Xiao \cite{Xiao09b,Xiao10} for related results in
some special cases.

\bigskip

\ack This paper was written while Li was visiting Department of
Statistics and Probability, Michigan State University (MSU) with
the support of China Scholarship Council (CSC) grant. Li thanks
MSU for the good working condition and CSC for the financial
support.

The authors thank the anonymous referee for pointing out a connection
between the operator-self-similarity in the present paper and the
group-self-similarity introduced by  Kolody\'{n}ski and Rosi\'{n}ski
\cite{KR03}. His/Her comments have helped to improve the manuscript
significantly.
\bigskip

\renewcommand{\baselinestretch}{1.0}

\vskip.4in

\begin{quote}
\begin{small}

\textsc{Yuqiang Li}: School of Finance and Statistics, East China Normal
University, Shanghai 200241, China.\\
E-mail: \texttt{yqli@stat.ecnu.edu.cn}\\
URL: \texttt{http://yqli.faculty.ecnu.cn}\\
[1mm]

\textsc{Yimin Xiao}: Department of Statistics
and Probability, A-413 Wells Hall, Michigan State
University, East Lansing, MI 48824, U.S.A.\\
E-mails: \texttt{xiao@stt.msu.edu}\\
URL: \texttt{http://www.stt.msu.edu/\~{}xiaoyimi}

\end{small}
\end{quote}

\end{document}